\documentclass[a4paper]{article}

\usepackage[latin1]{inputenc} 
\usepackage{graphicx}
\usepackage[pdftex,bookmarks]{hyperref}
\usepackage{amsmath}    
\usepackage{amssymb}
\usepackage{amsmath,amsthm}
\usepackage{textcomp}
\usepackage[all]{xy}
\usepackage{geometry}

\newtheorem{teo}{Theorem}[section]
\newtheorem{defi}{Definition}[section]
\newtheorem{lemma}{Lemma}[section]

\newtheorem{cor}{Corollary}[section]
\newtheorem{prop}{Proposition}[section]

\newtheorem{rem} {Remark}[section]

\DeclareMathOperator{\dvol}{dvol}

\DeclareMathOperator{\supp}{supp}

\DeclareMathOperator{\id}{Id}
\DeclareMathOperator{\vol}{vol}
\DeclareMathOperator{\reg}{reg}
\DeclareMathOperator{\sing}{sing}

\DeclareMathOperator{\Tr}{Tr}

\DeclareMathOperator{\abs}{abs}
\DeclareMathOperator{\rel}{rel}

\title{\huge \bf Degenerating Hermitian metrics, canonical bundle and spectral convergence}
\author{Francesco Bei  \bigskip \\
Dipartimento di Matematica, Sapienza Universit\`a di Roma,\\ E-mail addresses: \ bei@mat.uniroma1.it \     francescobei27@gmail.com }
\date{}

\begin{document}

\maketitle

\begin{abstract}
\noindent Let $(M,J)$ be a compact complex manifold of complex dimension $m$ and let $g_s$  be a one-parameter family of Hermitian forms on $M$ that are smooth and positive definite for each fixed $s\in (0,1]$ and that somehow degenerates to a Hermitian pseudometric $h$ for $s$ tending to $0$. In this paper under rather general assumptions on $g_s$ we prove various spectral convergence type theorems for the family of Hodge-Kodaira Laplacians $\Delta_{\overline{\partial},m,0,s}$ associated to $g_s$ and acting on the canonical bundle of $M$. In particular we show that, as $s$ tends to zero, the eigenvalues, the heat operators and the heat kernels corresponding to the family $\Delta_{\overline{\partial},m,0,s}$ converge to the eigenvalues, the heat operator and the heat kernel of $\Delta_{\overline{\partial},m,0,\mathrm{abs}}$, a suitable self-adjoint operator with entirely discrete spectrum defined on the limit space $(A,h|_A)$.\\
\end{abstract}



\tableofcontents

\section*{Introduction}
Whenever we have a sequence of Riemannian manifolds $\{(M_n,g_n)\}_{n\in \mathbb{N}}$ that somehow degenerate to a limit space $X$ it is a very interesting question to analyze the limit behavior of the spectral invariants, such as eigenvalues, eigenvectors, traces and so on, associated to the sequence $\{(M_n,g_n)\}_{n\in \mathbb{N}}$. This and related topics have been investigated in so many papers that even to report a representative sample of the literature is beyond the scope of this introduction. Just to mention few relevant works we can recall \cite{Anne}, \cite{ChDo}, \cite{Chavel}, \cite{CheegerColding}, \cite{Ding}, \cite{Fukaya}, \cite{Ji}, \cite{Kasue}, \cite{KaKu}, \cite{KuSh}, \cite{Lott} and \cite{Wolpert}. An interesting branch of this circle of ideas is the one that deals with  (real or complex) algebraic  varieties understood as the limit of a sequence of smooth algebraic varieties. In this kind of situation the behavior of the eigenvalues of the Laplace-Beltrami operator have been investigated in various papers, see for instance \cite{Gromov} which is devoted to real algebraic and semi-algebraic sets, \cite{Went} which deals with a smooth family of compact surfaces that degenerate to a surface with conical singularities  and \cite{KIY} which is concerned with a one-parameter degenerating family of projective algebraic manifolds in $\mathbb{C}\mathbb{P}^n$ over the unit disc. In this paper we tackle a somewhat similar problem but rather than the Laplace-Beltrami operator we are interested in the Hodge-Kodaira Laplacian acting on the sections of the canonical bundle. More precisely we are concerned with the following setting: $(M,J)$ is a compact complex manifold of complex dimension $m$ endowed with a Hermitian pseudometric $h$. We recall that a Hermitian pseudometric is nothing but a semipositive definite Hermitian product strictly positive on an open and dense subset $A\subset M$. As explained in the final part of this paper this is a rather general framework that encompasses for instance complex projective varieties endowed with the Fubini-Study metric and more generally compact and irreducible Hermitian complex spaces.  
Consider now the Dolbeault operator $\overline{\partial}_{m,0}:\Omega^{m,0}_c(A)\rightarrow \Omega^{m,1}_c(A)$ and let $\overline{\partial}_{m,0}^t:\Omega^{m,1}_c(A)\rightarrow \Omega^{m,0}_c(A)$ be the formal adjoint of $\overline{\partial}_{m,0}:\Omega^{m,0}_c(A)\rightarrow \Omega^{m,1}_c(A)$ with respect to $h|_A$. Finally consider $\overline{\partial}_{m,0}^t\circ \overline{\partial}_{m,0}:\Omega^{m,0}_c(A)\rightarrow \Omega^{m,0}_c(A)$, that is the Hodge-Kodaira Laplacian with respect to $h|_A$ acting on the canonical bundle of $A$. In \cite{FBei} several results concerning the $L^2$-spectral theory of  $\overline{\partial}_{m,0}^t\circ \overline{\partial}_{m,0}:\Omega^{m,0}_c(A)\rightarrow \Omega^{m,0}_c(A)$ were proved. In particular, by only requiring that $(A,g|_A)$ is parabolic with respect to some (and therefore all) Riemannian metric $g$ on $M$, we showed that given any closed extension $\mathfrak{\overline{d}}_{m,0}:L^2\Omega^{m,0}(A,h|_A)\rightarrow L^2\Omega^{m,1}(A,h|_A)$ of $\overline{\partial}_{m,0}:\Omega^{m,0}_c(A)\rightarrow \Omega^{m,1}_c(A)$  the operator 
$\mathfrak{\overline{d}}_{m,0}^*\circ \mathfrak{\overline{d}}_{m,0}:L^2\Omega^{m,0}(A,h|_A)\rightarrow L^2\Omega^{m,0}(A,h|_A)$ is a self-adjoint extension of $\overline{\partial}_{m,0}^t\circ \overline{\partial}_{m,0}:\Omega^{m,0}_c(A)\rightarrow \Omega^{m,0}_c(A)$ with entirely discrete spectrum, see \cite{FBei} Th. 4.1. Now let us introduce the product manifold $M\times [0,1]$ and let $p:M\times [0,1]\rightarrow M$ be the natural projection. Let $g_s$ be any measurable section of $p^*T^*M\otimes p^*T^*M\rightarrow M\times [0,1]$ such that
\begin{enumerate}
\item $g_s$ is a Hermitian metric on $M$ for each $s\in (0,1]$;
\item $g_s|_{A\times [0,1]}\in C^{\infty}(A\times [0,1],p^*T^*A\otimes p^*T^*A)$;
\item $g_0|_A=h|_A$;
\item $(A,g|_A)$ is parabolic with respect to some Riemannian metric $g$ on $M$;
\end{enumerate}
Roughly speaking  $g_s$ is a one-parameter family of Hermitian metrics on $M$ that on $A$ degenerates smoothly to a Hermitian pseudometric $h$ for $s\rightarrow 0$. 
For each $s\in (0,1]$ let $\Delta_{\overline{\partial},m,0,s}:L^2\Omega^{m,0}(M,g_s)\rightarrow L^2\Omega^{m,0}(M,g_s)$ be the unique closed (and therefore self-adjoint) extension of the Hodge-Kodaira Laplacian, with respect to the metric $g_s$, acting on the canonical bundle of $M$. It is well known by elliptic theory on compact manifolds that $\Delta_{\overline{\partial},m,0,s}:L^2\Omega^{m,0}(M,g_s)\rightarrow L^2\Omega^{m,0}(M,g_s)$ has entirely discrete spectrum, see e.g. \cite{Gilkey}.  We have finally all the ingredients to formulate the first main question addressed by this paper:\\

\noindent Let $s\in (0,1]$ and let $\{\lambda_k(s)\}_{k\in \mathbb{N}}$ be the eigenvalues of $\Delta_{\overline{\partial},m,0,s}:L^2\Omega^{m,0}(M,g_s)\rightarrow L^2\Omega^{m,0}(M,g_s)$. Under what assumptions on $g_s$ does $\lambda_k(s)\rightarrow \lambda_k(0)$ as $s\rightarrow 0$, where $\{\lambda_k(0)\}_{k\in \mathbb{N}}$ are the eigenvalues of $\mathfrak{\overline{d}}_{m,0}^*\circ \mathfrak{\overline{d}}_{m,0}:L^2\Omega^{m,0}(A,h|_A)\rightarrow L^2\Omega^{m,0}(A,h|_A)$ and  $\mathfrak{\overline{d}}_{m,0}:L^2\Omega^{m,0}(A,h|_A)\rightarrow L^2\Omega^{m,1}(A,h|_A)$ is a suitable closed extension of $\overline{\partial}_{m,0}:\Omega^{m,0}_c(A)\rightarrow \Omega^{m,1}_c(A)$?\\

\noindent In the first main result of this paper we provide a positive answer to the above question by requiring in addition that there exists a positive constant $\nu$ such that
$\frac{1}{\nu}h\leq g_s\leq \nu g_1$ for each $s\in (0,1]$. More precisely we have: 
\begin{teo}
\label{specspec}
Let $g_s$ be any measurable section of $p^*T^*M\otimes p^*T^*M\rightarrow M\times [0,1]$  that satisfies the fourth properties listed above. Assume moreover that there exists a positive constant $\nu\in \mathbb{R}$ such that $\frac{1}{\nu}\leq g_s\leq \nu g_1$ for each $s\in (0,1]$. Let
\begin{equation}
\label{mimi}
\Delta_{\overline{\partial},m,0,\abs}:L^2\Omega^{m,0}(A,h|_A)\rightarrow L^2\Omega^{m,0}(A,h|_A)
\end{equation}
be the operator defined as $\Delta_{\overline{\partial},m,0,\abs}:=\overline{\partial}_{m,0,\max}^*\circ\overline{\partial}_{m,0,\max}$ where 
\begin{equation}
\label{mumu}
\overline{\partial}_{m,0,\max}:L^2\Omega^{m,0}(A,h|_A)\rightarrow L^2\Omega^{m,1}(A,h|_A)
\end{equation}
is the maximal extension of $\overline{\partial}_{m,0}:\Omega_c^{m,0}(A)\rightarrow \Omega_c^{m,1}(A)$ and 
$\overline{\partial}_{m,0,\max}^*:L^2\Omega^{m,1}(A,h|_A)\rightarrow L^2\Omega^{m,0}(A,h|_A)$ is the adjoint of \eqref{mumu}. 
For each $s\in (0,1]$ let $0\leq\lambda_1(s)\leq\lambda_2(s)\leq...\leq \lambda_k(s)\leq...$ be the eigenvalues of 
\begin{equation}
\label{gurugu}
\Delta_{\overline{\partial},m,0,s}:L^2\Omega^{m,0}(M,g_s)\rightarrow L^2\Omega^{m,0}(M,g_s)
\end{equation}
and let $0\leq \lambda_1(0)\leq \lambda_2(0)\leq...\leq \lambda_k(0)\leq...$ be the eigenvalues of \eqref{mimi}. Then  $$\lim_{s\rightarrow 0} \lambda_k(s)=\lambda_k(0)$$ for each positive integer $k$.\\ Moreover let $\{s_n\}_{n\in \mathbb{N}}\subset (0,1]$ be any sequence such that $s_n\rightarrow 0$ as $n\rightarrow \infty$ and let  $\{\eta_1(s_n),\eta_2(s_n),.$ $..,\eta_k(s_n),...\}$ be  any orthonormal basis of  $L^2\Omega^{m,0}(M,g_{s_n})$  made by eigenforms of \eqref{gurugu} with corresponding eigenvalues  $\{\lambda_1(s_n),$ $...,\lambda_k(s_n),...\}$. Then there exists a subsequence $\{z_n\}\subset \{s_n\}$ and an orthonormal basis $\{\eta_1(0),\eta_2(0),...,\eta_k(0),...\}$ of $L^2\Omega^{m,0}(A,h|_A)$ made by eigenforms of \eqref{mimi} with corresponding eigenvalues $\{\lambda_1(0),...,\lambda_k(0),...\}$ such that 
$$\lim_{n\rightarrow \infty} \eta_k(z_n)=\eta_k(0)$$ in $L^2\Omega^{m,0}(A,h|_A)$ for each positive integer $k$.
\end{teo}

Since we have the convergence of the eigenvalues it is natural to investigate if there is convergence of more sophisticated ``spectral objects''. This task is tackled in the second main result of this paper where the limit behavior of the corresponding heat operators is studied. More precisely let $e^{-t\Delta_{\overline{\partial},m,0,\mathrm{abs}}}:L^2\Omega^{m,0}(A,h|_A)\rightarrow L^2\Omega^{m,0}(A,h|_A)$ and $e^{-t\Delta_{\overline{\partial},m,0,s}}:L^2\Omega^{m,0}(M,g_s)\rightarrow L^2\Omega^{m,0}(M,g_s)$  be the heat operators associated to \eqref{mimi} and \eqref{gurugu}, respectively. These are all trace-class operators. When $s\in (0,1]$ it is again a classical result of elliptic theory on compact manifolds whereas for $e^{-t\Delta_{\overline{\partial},m,0,\mathrm{abs}}}$ it is proved in \cite[Cor 4.2]{FBei}. By the fact that $L^2\Omega^{m,0}(A,h|_A)=L^2\Omega^{m,0}(M,g_s)$ for each $s\in (0,1]$, see \eqref{stesso}, we can look at $e^{-t\Delta_{\overline{\partial},m,0,s}}:L^2\Omega^{m,0}(M,g_s)\rightarrow L^2\Omega^{m,0}(M,g_s)$ as a family of trace-class operators acting on a fixed Hilbert space. It is therefore natural to investigate the limit behavior of 
$e^{-t\Delta_{\overline{\partial},m,0,s}}:L^2\Omega^{m,0}(M,g_s)\rightarrow L^2\Omega^{m,0}(M,g_s)$ with respect to the trace-class norm wondering in particular if $e^{-t\Delta_{\overline{\partial},m,0,s}}$ converges to $e^{-t\Delta_{\overline{\partial},m,0,\mathrm{abs}}}$. This is the goal of our second main result that indeed shows  that $e^{-t\Delta_{\overline{\partial},m,0,s}}$ converges to $e^{-t\Delta_{\overline{\partial},m,0,\mathrm{abs}}}$ as $s\rightarrow 0$ with respect to the trace-class norm. More precisely
\begin{teo}
\label{guanciale}
Let $t_0\in (0,\infty)$ be arbitrarily fixed. Then $$\lim_{s\rightarrow0}\sup_{t\in [t_0,\infty)}\Tr|e^{-t\Delta_{\overline{\partial},m,0,s}}-e^{-t\Delta_{\overline{\partial},m,0,\mathrm{abs}}}|=0.$$
 Equivalently $e^{-t\Delta_{\overline{\partial},m,0,s}}$ converges to $e^{-t\Delta_{\overline{\partial},m,0,\mathrm{abs}}}$  as $s\rightarrow 0$ with respect to the trace-class norm and uniformly on $[t_0,\infty)$.
\end{teo}
We stress on the fact that our results  require neither assumptions on the dimension of $M$ nor  restrictions on the curvature of $g_s$. Moreover we do not need to impose any particular asymptotic to $h$ near $Z$.\\
Now we continue this introduction by describing how the paper is sort out. The first section contains the background material. In the second section we recall some results of functional analysis that play a key role in the proof of Th. \ref{specspec}. In particular we recall the notion of Mosco convergence, introduced originally in \cite{Mosco} and later generalized in \cite{KuSh}, as we found this machinery very suitable to prove Th. \ref{specspec}. The third  section is devoted to the main results of this paper. Besides the  proofs of Th. \ref{specspec} and Th. \ref{guanciale}  it contains further results and applications. In particular  a converge theorem for the heat kernels  of the family  $e^{-t\Delta_{\overline{\partial},m,0,s}}$ to the heat kernel of  $e^{-t\Delta_{\overline{\partial},m,0,\mathrm{abs}}}$ is derived, see Th. \ref{HilSCon}, and moreover some applications to the corresponding family of zeta functions are given, see Th.  \ref{Zeta}. The fourth and last section contains some examples and applications.\\
Finally we conclude this introduction with the following remark. The reader may wonder why we are concerned only with the Hodge-Kodaira Laplacian acting on the canonical bundle. Besides the well known importance played by the canonical bundle in complex geometry there is another, more technical, reason. To our best knowledge, without requiring restrictive condition either on the dimension of $X$ or on $\sing(X)$, there are only two cases where the Hodge-Kodaira Laplacian acting on the regular part of a compact Hermitian complex space  is known to have  a self-adjoint extension with entirely discrete spectrum: the scalar case, first proved in \cite{LT} and later generalized in \cite{FraB}, and the case of the canonical bundle \cite{FBei}. This paper is devoted to the latter case.\\

\noindent\textbf{Acknowledgements}. This paper was mainly written while the author was a postdoc at the Mathematics Department of the University of Padova. He wishes to thank that institution for financial support. Moreover the author wishes to thank the anonymous referee for helpful comments.

\section{Background material}

This section is devoted to the background material. In the first part we recall some basic notions on closed extensions of differential operators whereas the second part is concerned with some properties of Hermitian metrics. 
Let  $(M,J,g)$ be a complex Hermitian manifold of real dimension $2m$.  As usual with $\Lambda^{a,b}(M)$ we denote the bundle of $(a,b)$-forms, that is $\Lambda^a(T^{1,0,*}M)\otimes \Lambda^b(T^{0,1,*}M)$ and by $\Omega^{a,b}(M)$, $\Omega^{a,b}_c(M)$ we denote respectively the space of sections, sections with compact support,  of $\Lambda^{a,b}(M)$. On the bundle $\Lambda^{a,b}(M)$ we consider the Hermitian metric induced by $g$ and  we  label it by $g^*_{a,b}$. With $L^2\Omega^{a,b}(M,g)$ we denote the space of measurable $(a,b)$-forms $\eta$ such that $\int_Mg^*_{a,b}(\eta,\eta)\dvol_g<\infty$ where $\dvol_g$ is the volume form induced by $g$. This is a Hilbert space  whose inner product is given by $$\langle\eta,\omega\rangle_{L^2\Omega^{a,b}(M,g)}=\int_Mg^*_{a,b}(\eta,\omega)\dvol_g$$ for any $\eta,\omega\in L^2\Omega^{a,b}(M,g)$. The Dolbeault operator acting on $(a,b)$-forms is labeled by $\overline{\partial}_{a,b}:\Omega^{a,b}(M)\rightarrow \Omega^{a,b+1}(M)$  while $\overline{\partial}_{a,b}^t:\Omega^{a,b+1}(M)\rightarrow \Omega^{a,b}(M)$ denotes the formal adjoint of $\overline{\partial}_{a,b}:\Omega^{a,b}(M)\rightarrow \Omega^{a,b+1}(M)$ with respect to $g$. We look at $\overline{\partial}_{a,b}:L^2\Omega^{a,b}(M,g)\rightarrow L^2\Omega^{a,b+1}(M,g)$ as an unbounded and densely defined operator defined on $\Omega_c^{a,b}(M)$ and we denote by $\overline{\partial}_{a,b,\max/\min}:L^2\Omega^{a,b}(M,g)\rightarrow L^2\Omega^{a,b+1}(M,g)$  its maximal and minimal closed extension, respectively. We recall that the maximal  closed extension is defined in the distributional sense: $\omega \in \mathcal{D}(\overline{\partial}_{a,b,\max})$ and $\overline{\partial}_{a,b,\max}\omega=\eta\in L^2\Omega^{a,b+1}(M,g)$ if $\langle\omega,\overline{\partial}_{a,b}^t\phi\rangle_{L^2\Omega^{a,b}(M,g)}=\langle\eta,\phi\rangle_{L^2\Omega^{a,b+1}(M,g)}$ for every $\phi\in \Omega^{a,b+1}_c(M)$. The minimal closed extension  is defined as the graph closure of $\overline{\partial}_{a,b}$, that is $\omega \in \mathcal{D}(\overline{\partial}_{a,b,\min})$ and $\overline{\partial}_{a,b,\min}\omega=\eta\in L^2\Omega^{a,b+1}(M,g)$ if there exists $\{\phi_k\}_{k\in \mathbb{N}}\subset \Omega^{a,b}_c(M)$ such that $\phi\rightarrow \omega$ and $\overline{\partial}_{a,b}\phi\rightarrow \eta$ as $n\rightarrow \infty$ in $L^2\Omega^{a,b}(M,g)$ and $L^2\Omega^{a,b+1}(M,g)$, respectively. In analogous way we can define $\overline{\partial}_{a,b,\max/\min}^t:L^2\Omega^{a,b+1}(M,g)\rightarrow L^2\Omega^{a,b}(M,g)$, that is the maximal/minimal closed extension of $\overline{\partial}_{a,b}^t:\Omega^{a,b+1}_c(M)\rightarrow \Omega^{a,b}_c(M)$, respectively. It is easy to check that $\overline{\partial}_{a,b,\max/\min}^t=\overline{\partial}_{a,b,\min/\max}^*$, that is $\overline{\partial}_{a,b,\max}^t:L^2\Omega^{a,b+1}(M,g)\rightarrow L^2\Omega^{a,b}(M,g)$ is the adjoint of $\overline{\partial}_{a,b,\min}:L^2\Omega^{a,b}(M,g)\rightarrow L^2\Omega^{a,b+1}(M,g)$ whereas $\overline{\partial}_{a,b+1,\min}^t:L^2\Omega^{a,b+1}(M,g)\rightarrow L^2\Omega^{a,b}(M,g)$ is the adjoint of $\overline{\partial}_{a,b,\max}:L^2\Omega^{a,b}(M,g)\rightarrow L^2\Omega^{a,b+1}(M,g)$. Consider now the Hodge-Kodaira Laplacian  $$\Delta_{\overline{\partial},a,b}:\Omega^{a,b}_c(M)\rightarrow \Omega^{a,b}_c(M),\ \Delta_{\overline{\partial},a,b}:=\overline{\partial}_{a,b-1}\circ\overline{\partial}^t_{a,b-1}+\overline{\partial}_{a,b}^t\circ \overline{\partial}_{a,b}.$$
 We recall the definition of  two important self-adjoint extensions of $\Delta_{\overline{\partial},a,b}$:
\begin{equation}
\label{asdf}
\overline{\partial}_{a,b-1,\max}\circ \overline{\partial}_{a,b-1,\min}^t+\overline{\partial}_{a,b,\min}^t\circ \overline{\partial}_{a,b,\max}:L^2\Omega^{a,b}(M,g)\rightarrow L^2\Omega^{a,b}(M,g)
\end{equation} 
and 
\begin{equation}
\label{buio}
\overline{\partial}_{a,b-1,\min}\circ \overline{\partial}_{a,b-1,\max}^t+\overline{\partial}_{a,b,\max}^t\circ \overline{\partial}_{a,b,\min}:L^2\Omega^{a,b}(M,g)\rightarrow L^2\Omega^{a,b}(M,g)
\end{equation}
called respectively the absolute and the relative extension. The operator \eqref{asdf}, the absolute extension, is denoted with  $\Delta_{\overline{\partial},a,b,\abs}$ and its domain is defined as $$\mathcal{D}(\Delta_{\overline{\partial},a,b,\abs})=\left\{\omega\in \mathcal{D}(\overline{\partial}_{a,b,\max})\cap \mathcal{D}(\overline{\partial}_{a,b-1,\min}^t):\overline{\partial}_{a,b,\max}\omega \in \mathcal{D}(\overline{\partial}^t_{a,b,\min})\ \text{and}\ \overline{\partial}_{a,b-1,\min}^t\omega \in \mathcal{D}(\overline{\partial}_{a,b-1,\max})\right\}.$$
The operator \eqref{buio}, the relative extension,  is denoted with  $\Delta_{\overline{\partial},a,b,\rel}$ and its domain is defined as $$\mathcal{D}(\Delta_{\overline{\partial},a,b,\rel})=\left\{\omega\in \mathcal{D}(\overline{\partial}_{a,b,\min})\cap \mathcal{D}(\overline{\partial}_{a,b-1,\max}^t):\overline{\partial}_{a,b,\min}\omega \in \mathcal{D}(\overline{\partial}^t_{a,b,\max})\ \text{and}\ \overline{\partial}_{a,b-1,\max}^t\omega \in \mathcal{D}(\overline{\partial}_{a,b-1,\min})\right\}.$$ This concludes the first part of this introduction. For more details  we refer to \cite{FraBei} and the reference therein. Now we recall some background material concerning Hermitian metrics. These properties are certainly well known to the experts. However it is not easy to find them in the literature. Therefore we preferred to write them down believing that this could be helpful for the unfamiliar reader. The proof are omitted because they lie on elementary arguments of linear algebra. Let $(M,J)$ be a complex manifold of complex dimension $m$ and let $g$ and $h$ be Hermitian metrics on $M$. Then there exists $F\in C^{\infty}(M,\mathrm{End}(TM))$ such that $h(\cdot,\cdot)=g(F\cdot,\cdot)$. It is immediate to verify that $F$ and $J$ commute. For any $p\in M$ consider $F_p:T_pM\rightarrow T_pM$ and $J_p:T_pM\rightarrow T_pM$. As $F_p\circ J_p=J_p\circ F_p$ every eigenspace of $F_p$ is preserved by $J_p$ and therefore it has even dimension. This tells us that the eigenvalues of $F_p$ are given by $\{\lambda_1(p),\lambda_1(p),\lambda_2(p),\lambda_2(p),...,\lambda_m(p),\lambda_m(p)\}$ with $0<\lambda_1(p)\leq\lambda_2(p)\leq...\leq \lambda_m(p)$. Moreover if  $E_p$ is an eigenspace of $F_p:T_pM\rightarrow T_pM$ of dimension $2k$ then, for any $k$-tuple of linearly independent eigenvectors $v_1,...,v_k\in E_p$, the set $\{v_1,J_pv_1,...,v_k,J_pv_k\}$ is a base for $E_p$. Let $F_{\mathbb{C}}\in C^{\infty}(M,\mathrm{End}(TM\otimes \mathbb{C}))$ be the $\mathbb{C}$-linear endomorphism induced by $F$ on the complexified tangent bundle. Then the eigenvalues of $F_{\mathbb{C},p}$ are still $\{\lambda_1(p),\lambda_1(p),\lambda_2(p),\lambda_2(p),...,\lambda_m(p),\lambda_m(p)\}$ with corresponding eigenspaces obtained by complexification of the eigenspaces of $F_p:T_pM\rightarrow T_pM$. Consider an arbitrary eigenspace $E_p$ of $F_p$. Then $E_{\mathbb{C},p}:=E_p\otimes \mathbb{C}$ splits as $E_{\mathbb{C},p}=E_{\mathbb{C},p}^{1,0}\oplus E_{\mathbb{C},p}^{0,1}$ with $E_{\mathbb{C},p}^{1,0}=T^{1,0}_pM\cap E_{\mathbb{C},p}$ and $E_{\mathbb{C},p}^{0,1}=T^{0,1}_pM\cap E_{\mathbb{C},p}$. Moreover it is easy to check that both $T^{1,0}M$ and $T^{0,1}M$ are preserved by $F_{\mathbb{C}}$. If we define $F^{1,0}_{\mathbb{C}}:=F_{\mathbb{C}}|_{T^{1,0}M}$ and $F^{0,1}_{\mathbb{C}}:=F_{\mathbb{C}}|_{T^{0,1}M}$ then, for any $p\in M$, the eigenvalues of $F_{\mathbb{C},p}^{1,0}$ are $\{\lambda_1(p),\lambda_2(p),...,\lambda_{m-1}(p),\lambda_m(p)\}$ with  eigenspaces given by the $(1,0)$-part of the complexification of the corresponding eigenspaces of $F_p$. In particular if $E_p$ is any eigenspace of $F_p$ of real dimension $2k$ with a base of eigenvectors given by  $\{v_1,J_pv_1,...,v_k,J_pv_k\}$
then $E_{\mathbb{C},p}^{1,0}$ becomes a complex $k$-dimensional eigenspace of $F_{\mathbb{C},p}^{1,0}$  with a  base of  eigenvectors given by  $\{v_1-iJ_pv_1,...,v_k-iJ_pv_k\}$. Analogously the eigenvalues of $F_{\mathbb{C},p}^{0,1}$ are $\{\lambda_1(p),\lambda_2(p),...,\lambda_{m-1}(p),\lambda_m(p)\}$ with  eigenspaces given by the $(0,1)$-part of the complexification of the corresponding eigenspaces of $F_p$. In particular if $E_p$ is any eigenspace of $F_p$ of real dimension $2k$ with a base of eigenvectors given by  $\{v_1,J_pv_1,...,v_k,J_pv_k\}$ then $E_{\mathbb{C},p}^{0,1}$ becomes a complex $k$-dimensional eigenspace of $F_{\mathbb{C},p}^{0,1}$  with a base of  eigenvectors given by  $\{v_1+iJ_pv_1,...,v_k+iJ_pv_k\}$. As a first consequence we can deduce that: $$\det(F_p)=\det(F_{\mathbb{C},p})=\det(F_{\mathbb{C},p}^{1,0})\det(F_{\mathbb{C},p}^{0,1})=(\det(F_{\mathbb{C},p}^{1,0}))^2=(\det(F_{\mathbb{C},p}^{0,1}))^2.$$
Let now $g^*$ and $h^*$ be the Hermitian  metrics induced by $g$ and $h$ on $T^*M$, respectively. It is easy to verify that  $h^*(\cdot,\cdot)=g^*((F^{-1})^t\cdot,\cdot)$ where $(F^{-1})^t$ is the transpose of $F^{-1}$, that is the endomorphism of $T^*M$ induced by $F^{-1}$. Let us define $G\in C^{\infty}(M,\mathrm{End}(T^*M))$ as $G:=(F^{-1})^t$. Then the eigenvalues of $G_p$ are $\{\frac{1}{\lambda_1(p)},\frac{1}{\lambda_1(p)},...,\frac{1}{\lambda_m(p)},\frac{1}{\lambda_m(p)}\}$. Likewise the case of the tangent bundle, with self-explanatory notation, we introduce $G_{\mathbb{C}}$, $G_{\mathbb{C}}^{1,0}$ and $G_{\mathbb{C}}^{0,1}$ acting on $T^*M\otimes \mathbb{C}$,  $T^{1,0,*}M$ and $T^{0,1,*}M$, respectively. The eigenvalues of  both $G_{\mathbb{C},p}^{1,0}$ and $G_{\mathbb{C},p}^{0,1}$ are $\{\frac{1}{\lambda_1(p)},\frac{1}{\lambda_2(p)},...,\frac{1}{\lambda_{m-1}(p)},\frac{1}{\lambda_m(p)}\}$. In particular we have $$\det(G_{\mathbb{C},p}^{1,0})=(\det(F_{\mathbb{C},p}^{1,0}))^{-1}\ \text{and}\  \det(G_{\mathbb{C},p}^{0,1})=(\det(F_{\mathbb{C},p}^{0,1}))^{-1}.$$ Let us now label with $g_{\mathbb{C}}$ and $h_{\mathbb{C}}$ the Hermitian metrics on $TM\otimes \mathbb{C}$ induced by $g$ and $h$, respectively. We recall that for any $p\in M$, $u,v\in T_pM$ and $\alpha,\beta\in \mathbb{C}$ we have $g_{\mathbb{C}}(u\otimes \alpha,v\otimes \beta)=\alpha\overline{\beta}g(u,v)$ and $h_{\mathbb{C}}(u\otimes \alpha,v\otimes \beta)=\alpha\overline{\beta}h(u,v)$. Let $h^*_{\mathbb{C}}$, $h^*_{a,b}$, $g^*_{\mathbb{C}}$ and $g_{a,b}^*$ be the Hermitian metrics on $T^*M\otimes \mathbb{C}$ and $\Lambda^{a,b}(M)$ induced by $h_{\mathbb{C}}$ and $g_{\mathbb{C}}$, respectively. It is easy to verify that $h^*_{a,b}=h^*_{a,0}\otimes h^*_{0,b}$, $g^*_{a,b}=g^*_{a,0}\otimes g^*_{0,b}$ and that $h^*_{a,0}(\cdot,\cdot)=g_{a,0}^*(G_{\mathbb{C}}^{a,0}\cdot,\cdot)$, $h^*_{0,b}(\cdot,\cdot)=g_{0,b}^*(G_{\mathbb{C}}^{0,b}\cdot,\cdot)$, $h^*_{a,b}(\cdot, \cdot)=g^*_{a,b}(G_{\mathbb{C}}^{a,0}\otimes G_{\mathbb{C}}^{0,b}\cdot,\cdot)$ where $G_{\mathbb{C}}^{0,b}\in C^{\infty}(M,\mathrm{End}(\Lambda^{0,b}(M)))$ and $G_{\mathbb{C}}^{a,0}\in C^{\infty}(M,\mathrm{End}(\Lambda^{a,0}(M)))$ are the endomorphisms induced in the natural way by $G^{0,1}_{\mathbb{C}}$ and $G^{1,0}_{\mathbb{C}}$, respectively.  Let $\omega\in \Omega^{0,b}_c(M)$. Then for the $L^2$-inner product we have 
\begin{align}
\nonumber & \langle\omega,\omega\rangle_{L^2\Omega^{0,b}(M,h)}=\int_Mh^*_{0,b}(\omega,\omega)\dvol_h=\int_Mg^*_{0,b}(G^{0,b}_{\mathbb{C}}\omega,\omega)\sqrt{\det(F)}\dvol_g\leq \int_M|G^{0,b}_{\mathbb{C}}|_{g^*_{0,b}}g^*_{0,b}(\omega,\omega)\sqrt{\det(F)}\dvol_g
\end{align}
where $|G^{0,b}_{\mathbb{C}}|_{g^*_{0,b}}:M\rightarrow \mathbb{R}$ is the function that assigns to each $p\in M$ the pointwise operator norm of $G^{0,b}_{\mathbb{C},p}:\Lambda^{0,b}_p(M)\rightarrow \Lambda^{0,b}_p(M)$ with respect to $g_{0,b}^*$, that is 
\begin{equation}
\label{pointnorm}
|G^{0,b}_{\mathbb{C}}|_{g^*_{0,b}}^2(p)=\sup_{0\neq v\in \Lambda^{0,b}_p(M)}\frac{g^*_{0,b}(G^{0,b}_{\mathbb{C}}v,G^{0,b}_{\mathbb{C}}v)}{g^*_{0,b}(v,v)}.
\end{equation}
In particular, if $|G^{0,b}_{\mathbb{C}}|_{g^*_{0,b}}\sqrt{\det(F)}\in L^{\infty}(M)$ then we have 
\begin{align}
\nonumber & \langle\omega,\omega\rangle_{L^2\Omega^{0,b}(M,h)}=\int_Mh^*_{0,b}(\omega,\omega)\dvol_h=\int_Mg^*_{0,b}(G^{0,b}_{\mathbb{C}}\omega,\omega)\sqrt{\det(F)}\dvol_g\leq \int_M|G^{0,b}_{\mathbb{C}}|_{g^*_{0,b}}g^*_{0,b}(\omega,\omega)\sqrt{\det(F)}\dvol_g\leq \\
\nonumber & \||G^{0,b}_{\mathbb{C}}|_{g^*_{0,b}}\sqrt{\det(F)}\|_{L^{\infty}(M)}\int_Mg^*_{0,b}(\omega,\omega)\dvol_g=\||G^{0,b}_{\mathbb{C}}|_{g^*_{0,b}}\sqrt{\det(F)}\|_{L^{\infty}(M)}\langle\omega,\omega\rangle_{L^2\Omega^{0,b}(M,g)}.
\end{align}
Consider now the case $(m,0)$. Let $\xi,\chi\in \Omega^{m,0}_c(M)$. Then we have 
\begin{align}
\label{equality} & \langle\xi,\chi\rangle_{L^2\Omega^{m,0}(M,h)}=\int_Mh^*_{m,0}(\xi,\chi)\dvol_h=\int_Mg^*_{m,0}(\det(G_{\mathbb{C}}^{1,0})\xi,\chi)\sqrt{\det(F)}\dvol_g=\\
\nonumber & \int_Mg^*_{m,0}(\xi,\chi)\dvol_g=\langle\xi,\chi\rangle_{L^2\Omega^{m,0}(M,g)}.
\end{align}
Hence we can conclude that we have an equality of Hilbert spaces $L^2\Omega^{m,0}(M,h)=L^2\Omega^{m,0}(M,g)$. Finally consider a form $\psi\in \Omega^{m,b}_c(M)$ with $b>0$. Then for the $L^2$-inner product we have 
\begin{align}
\label{ibra3} 
& \langle\psi,\psi\rangle_{L^2\Omega^{m,b}(M,h)}=\int_M h_{m,b}^*(\psi,\psi)\dvol_h= \int_M g_{m,b}^*(G^{m,0}_{\mathbb{C}}\otimes G^{0,b}_{\mathbb{C}}\psi,\psi)\sqrt{\det(F)}\dvol_g=\\
& \nonumber \int_M g_{m,b}^*(\det(G^{1,0}_{\mathbb{C}})\otimes G^{0,b}_{\mathbb{C}}\psi,\psi)\sqrt{\det(F)}\dvol_g= \int_M g_{m,b}^*( \id\otimes G^{0,b}_{\mathbb{C}}\psi,\psi)\dvol_g=\\
& \nonumber \int_M g_{m,0}^*\otimes g^*_{0,b}( \id\otimes G^{0,b}_{\mathbb{C}}\psi,\psi)\dvol_g\leq \int_M |G^{0,b}_{\mathbb{C}}|_{g^*_{0,b}}g_{m,0}^*\otimes g^*_{0,b}(\psi,\psi)\dvol_g=\int_M |G^{0,b}_{\mathbb{C}}|_{g^*_{0,b}} g^*_{m,b}(\psi,\psi)\dvol_g.
\end{align}

Thus  we can conclude that whenever $|G^{0,b}_{\mathbb{C}}|_{g^*_{0,b}}\in L^{\infty}(M)$, then 
\begin{equation}
\label{miumiu2}
\langle\psi,\psi\rangle_{L^2\Omega^{m,b}(M,h)}\leq \||G^{0,b}_{\mathbb{C}}|_{g^*_{0,b}}\|_{L^{\infty}(M)}\langle\psi,\psi\rangle_{L^2\Omega^{m,b}(M,g)}
\end{equation}
whereas if $g^*_{0,b}(G^{0,b}_{\mathbb{C}}\cdot,\cdot)\geq cg^*_{0,b}(\cdot,\cdot)$ for some positive  $c\in \mathbb{R}$,  then  $g_{m,0}^*\otimes g^*_{0,b}( \id\otimes G^{0,b}_{\mathbb{C}}\psi,\psi)\geq cg_{m,0}^*\otimes g^*_{0,b}( \psi,\psi)$ and therefore
\begin{equation}
\label{bibi2}
\langle\psi,\psi\rangle_{L^2\Omega^{m,b}(M,h)}\geq c\langle\psi,\psi\rangle_{L^2\Omega^{m,b}(M,g)}.
\end{equation}

\section{Functional analytic prerequisites}
In this section we recall briefly  some functional analytic tools that will be used later on. All the material is taken from \cite{KuSh}. We refer to it for an in-depth treatment.\\ Let $\{H_n\}_{n\in \mathbb{N}}$ be a sequence of infinite dimensional separable complex Hilbert spaces. Let $H$ be another infinite dimensional separable complex Hilbert space.  Let us label by $\langle \cdot ,\ \cdot \rangle_{H_n}$, $\|\cdot\|_{H_n}$, $\langle \cdot ,\ \cdot \rangle_{H}$ and $\|\cdot\|_{H}$  the corresponding inner products and norms. Let $\mathcal{C}\subseteq H$ be a dense subset. Assume that for every $n\in \mathbb{N}$ there exists a linear map $\Phi_n:\mathcal{C}\rightarrow H_n$. We will say that $H_n$ converges to $H$ as $n\rightarrow \infty$ if and only if
\begin{equation}
\label{limitHilbert}
\lim_{n\rightarrow\infty}\|\Phi_nu\|_{H_n}=\|u\|_{H}
\end{equation}
for any $u\in \mathcal{C}$.\\

\noindent \textbf{Assumption:} In the next definitions and propositions we will always assume that the sequence $\{H_n\}_{n\in \mathbb{N}}$ {\em converges} to $H$.

\begin{defi}
\label{strong}
Let $u\in H$ and let $\{u_n\}_{n\in \mathbb{N}}$ be a sequence such that $u_n\in H_n$ for each $n\in \mathbb{N}$. We say that $u_n$ strongly converges to $u$ as $n\rightarrow \infty$ if there exists a net $\{v_{\beta}\}_{\beta\in \mathcal{B}}\subset \mathcal{C}$ tending to $u$ in $H$ such that
\begin{equation}
\label{stronglimit}
\lim_{\beta} \limsup_{n\rightarrow\infty}\|\Phi_n v_{\beta}-u_n\|_{H_n}=0
\end{equation}
\end{defi}

We note that for any arbitrarily fixed $u\in \mathcal{C}$ the sequence $\{\Phi_nu\}_{n\in \mathbb{N}}$ strongly converges to $u$. This is an immediate consequence of \eqref{limitHilbert} and Def. \ref{strong}.

\begin{defi}
\label{weak}
Let $u\in H$ and let $\{u_n\}_{n\in \mathbb{N}}$ be a sequence such that $u_n\in H_n$ for each $n\in \mathbb{N}$. We say that $u_n$ weakly converges to $u$ as $n\rightarrow \infty$ if
\begin{equation}
\label{weaklimit}
\lim_{n\rightarrow \infty} \langle u_n,w_n\rangle_{H_n}=\langle u,w\rangle_{H}
\end{equation}
for any  $w\in H$ and any sequence $\{w_n\}_{n\in \mathbb{N}}$, $w_n\in H_n$, strongly convergent to $w$.
\end{defi}

\begin{prop}
\label{bounded}
Let $\{u_n\}_{n\in \mathbb{N}}$ be a sequence such that $u_n\in H_n$ for each $n\in \mathbb{N}$. Assume that there exists a positive real number $c$ such that $\|u_n\|_{H_n}\leq c$ for every $n\in \mathbb{N}$. Then there exists a subsequence  $\{u_m\}_{m\in \mathbb{N}}\subset \{u_n\}_{n\in \mathbb{N}}$, $u_m\in H_m$, weakly convergent to some element $u\in H$.
\end{prop}

\begin{proof}
See \cite{KuSh} Lemma 2.2.
\end{proof}

\begin{prop}
\label{wibounded}
Let $\{u_n\}_{n\in \mathbb{N}}$,  $u_n\in H_n$, be a sequence  weakly convergent to some element $u\in H$. Then there exists a positive real number $\ell$ such that 
\begin{equation}
\label{lerume}
\sup_{n\in \mathbb{N}} \|u_n\|_{H_n}\leq \ell \quad\quad\quad\ and\ \quad\quad\quad \|u\|_{H}\leq \liminf_{n\rightarrow \infty} \|u_n\|_{H_n}.
\end{equation}
Moreover $\{u_n\}_{n\in \mathbb{N}}$ converges strongly to $u$ if and only if $$\lim_{n\rightarrow \infty}\|u_n\|_{H_n}=\|u\|_{H}.$$
\end{prop}
\begin{proof}
See \cite{KuSh} Lemma 2.3.
\end{proof}

We have now the following remark. Consider the case of a constant sequence of infinite dimensional separable complex Hilbert spaces $\{H_n\}_{n\in \mathbb{N}}$, that is for each $n\in \mathbb{N}$ $H_n=H$, $\mathcal{C}=H$ and $\Phi_n:\mathcal{C}\rightarrow H_n$ is nothing but the identity $\id:H\rightarrow H$. Then Def. \ref{strong} and Def. \ref{weak} coincide with ordinary notions of convergence in $H$ and weak convergence in $H$. Indeed let $\{v_n\}\subset H$ be a sequence converging to some $v\in H$. Then by taking any constant net $\{v_{\beta}\}_{\beta\in B}\subset H$, $v_{\beta}:=v$ as a net in $H$ converging to $v$ we have $$\lim_{\beta}\limsup_{n\rightarrow \infty}\|\Phi_nv_{\beta}-v_n\|_{H_n}=\limsup_{n\rightarrow \infty}\|v-v_n\|_H=0.$$ Therefore $v_n\rightarrow v$ strongly in the sense of Def. \ref{strong}. Conversely let us assume that for some net $\{v_{\beta}\}_{\beta\in B}\subset H$ tending to $v$ in $H$ we have $$\lim_{\beta}\limsup_{n\rightarrow \infty}\|\Phi_nv_{\beta}-v_n\|_{H_n}=0.$$
Given any $\beta\in B$ we have $\|v-v_n\|_H\leq \|v-v_{\beta}\|_{H}+\|v_{\beta}-v_n\|_{H}$. Therefore for every $\beta\in B$ $$\limsup_{n\rightarrow \infty}\|v-v_n\|_H\leq \|v-v_{\beta}\|_{H}+\limsup_{n\rightarrow \infty}\|v_{\beta}-v_n\|_{H}$$ and finally 
$$\limsup_{n\rightarrow \infty}\|v-v_n\|_H\leq \lim_{\beta} \|v-v_{\beta}\|_{H}+\lim_{\beta}\limsup_{n\rightarrow \infty}\|v_{\beta}-v_n\|_{H}=0.$$ Therefore $v_n\rightarrow v$ in $H$ and thus we showed that Def. \ref{strong}  coincides with ordinary notion of convergence in $H$. Clearly this in turn implies immediately that Def. \ref{weak} coincides with the standard definition of weak convergence in $H$.\\

We recall now that a  quadratic form over a complex  Hilbert space $H$ is a sesquilinear form $Q:\mathcal{D}(Q)\times \mathcal{D}(Q)\rightarrow \mathbb{C}$, where $\mathcal{D}(Q)\subset H$ is a (not necessarily) dense linear subspace. Any quadratic form $Q$ in this paper is assumed to be nonnegative and Hermitian, i.e., $u\mapsto Q(u,v)$ is linear for any fixed $v\in \mathcal{D}(Q)$, $Q(u,v)=\overline{Q(v,u)}$, and $Q(u,u)\geq 0$ for any $u,v\in \mathcal{D}(Q)$. Clearly $Q_H(u,v):=\langle u,v\rangle_H+Q(u,v)$, $u,v\in \mathcal{D}(Q)$ becomes an inner product on $\mathcal{D}(Q)$. The form $Q$ is said to be {\em closed} if and only if $(\mathcal{D}(Q),Q_H)$ is a Hilbert space. Finally let us introduce $\overline{\mathbb{R}}:=\mathbb{R}\cup \{\infty\}$ and the functional $\overline{Q}:H\rightarrow \overline{\mathbb{R}}$ defined by 
$$
\overline{Q}(u)= \left\{
\begin{array}{ll}
Q(u,u) & u\in \mathcal{D}(Q)\\
\infty & u\in H\setminus \mathcal{D}(Q)
\end{array}
\right.
$$

The next definition, which is taken from \cite{KuSh}, extends to the case of a sequence of Hilbert spaces the notion of Mosco-convergence, originally formulated in \cite{Mosco} in the setting of a fixed Hilbert space.

\begin{defi}
\label{Mosco}
Consider a sequence of closed quadratic forms $\{Q_n\}_{n\in \mathbb{N}}$ such that $\mathcal{D}(Q_n)\subset H_n$ for any $n\in \mathbb{N}$. Let $Q$ be a closed quadratic form on $H$. We say that $\{Q_n\}_{n\in \mathbb{N}}$ Mosco-converges to $Q$ if: 
\begin{enumerate}
\item for each sequence $\{u_n\}_{n\in \mathbb{N}}$, $u_n\in H_n$, weakly convergent to some $u\in H$ we have $$\overline{Q}(u)\leq \liminf_{n\rightarrow \infty}\overline{Q}_n(u,u)$$
\item  for each $u\in H$ there exists a sequence $\{u_n\}_{n\in \mathbb{N}}$, $u_n\in H_n$, strongly convergent to $u$ such that $$\overline{Q}(u)=\lim_{n\rightarrow \infty} \overline{Q}_n(u).$$
\end{enumerate}
\end{defi}

\begin{defi}
Consider a sequence of closed quadratic forms $\{Q_n\}_{n\in \mathbb{N}}$ such that $\mathcal{D}(Q_n)\subset H_n$ for each $n\in \mathbb{N}$. The sequence is said to be  asymptotically compact if for any sequence $\{u_n\}_{n\in\mathbb{N}}$ with $u_n\in H_n$ and $$\limsup_{n\rightarrow \infty}\bigl(\|u_n\|_{H_n}+\overline{Q}_n(u)\bigr)<\infty$$ there exists a subsequence $\{u_m\}_{m\in \mathbb{N}}$, $u_m\in H_m$, that converges strongly to some $u\in H$.
\end{defi}

\begin{defi}
Consider a sequence of closed quadratic forms $\{Q_n\}_{n\in \mathbb{N}}$ such that $\mathcal{D}(Q_n)\subset H_n$ for each $n\in \mathbb{N}$. Let $Q$ be a closed quadratic form on $H$. We say that $\{Q_n\}_{n\in\mathbb{N}}$  compactly converges  to $Q$ if: 
\begin{enumerate}
\item $\{Q_n\}_{n\in \mathbb{N}}$ Mosco-converges to $Q$,
\item  $\{Q_n\}_{n\in \mathbb{N}}$ is asymptotically compact.
\end{enumerate}
\end{defi}
Consider now an unbounded, non-negative and densely defined self-adjoint operator $A:H\rightarrow H$. Let $Q_A$ be the closed quadratic form associated to $A$. For the general construction we refer to \cite{MaMa} pag. 377. Here we only recall that if $A=B^*\circ B$ with $\mathcal{D}(A)=\{u\in \mathcal{D}(B)\ \text{such that}\ Bu\in \mathcal{D}(B^*)\}$, where $B:H\rightarrow K$ is a closed and densely defined operator acting between $H$ and another separable Hilbert space $K$ and $B^*:K\rightarrow H$ is its adjoint, then $\mathcal{D}(Q_A)=\mathcal{D}(B)$ and $Q_A(u,v)=\langle Bu,Bv\rangle_K$ for any $u,v\in \mathcal{D}(Q_A)$. We have now the following important result.
\begin{teo}
\label{mainweapon}
Let $\{A_n\}_{n\in \mathbb{N}}$ be a sequence of unbounded, non-negative and densely defined self-adjoint operators $A_n:H_n\rightarrow H_n$. Let $A:H\rightarrow H$ be an unbounded, non-negative and densely defined self-adjoint operator. Assume that
\begin{itemize}
\item $A_n:H_n\rightarrow H_n$ has entirely discrete spectrum for each $n\in \mathbb{N}$,
\item the sequence of closed quadratic form $\{Q_{A_n}\}_{n\in \mathbb{N}}$ compactly converges to $Q_A$.
\end{itemize}
Then we have the following properties:
\begin{enumerate}
\item $A:H\rightarrow H$ has entirely discrete spectrum.
\item Let $0\leq \lambda_1(n)\leq \lambda_2(n)\leq ...\leq\lambda_k(n)\leq...$ be the eigenvalues of $A_n:H_n\rightarrow H_n$. Let $0\leq \lambda_1\leq \lambda_2\leq...\leq\lambda_k\leq...$ be the eigenvalues of $A:H\rightarrow H$. Then $$\lim_{n\rightarrow \infty}\lambda_k(n)=\lambda_k.$$
\item For each $n\in \mathbb{N}$ let $\{u_1(n),u_2(n),...,u_k(n),...\}$ be any orthonormal basis of $H_n$ made by eigenvectors of $A_n$ with corresponding eigenvalues $\{\lambda_1(n),\lambda_2(n),...,\lambda_k(n),...\}$. Then there exists a subsequence $\{H_m\}_{m\in \mathbb{N}}$ of $\{H_n\}_{n\in \mathbb{N}}$ and an orthonormal basis of $H$, $\{u_1,u_2,...,u_k,...\}$,  made by eigenvectors of $A$ with corresponding eigenvalues $\{\lambda_1,\lambda_2,...,\lambda_k,...\}$ such that $\{u_k(m)\}_{m\in \mathbb{N}}$ strongly converges to $u_k$ for any $k=1,2,...$ .
\end{enumerate}
\end{teo}
\begin{proof}
The first property is proved in \cite{KuSh} Cor. 2.4. The other properties are proved in \cite{KuSh} Cor. 2.5.
\end{proof}

\section{Spectral convergence for degenerating Hermitian metrics}
This section contains the main results of this paper. We start by introducing the setting and the notation. Let $(M,J)$ be a compact complex manifold of complex dimension $m$. Let $p:M\times [0,1]\rightarrow M$ be the canonical projection, let $g_s$ be a measurable section of $p^*T^*M\otimes p^*T^*M\rightarrow M\times[0,1]$ and let $h$ be a smooth, positive semidefinite Hermitian product on $M$ strictly positive on $A$, with $A\subset M$ open and dense. We make the following assumptions on $g_s$, $h$ and $A$: 
\begin{enumerate}
\item $g_s|_{A\times [0,1]}\in C^{\infty}(A\times [0,1],p^*T^*A\otimes p^*T^*A)$;
\item $g_0|_A=h|_A$;
\item For each fixed $s\in (0,1]$, $g_s$ is a smooth Hermitian metric on $M$;
\item  $(A,g_1|_A)$ is parabolic;
\item There exists a positive constant $\nu$ such that on $M$  we have $$\frac{1}{\nu}h\leq g_s\leq \nu g_1$$ for each $s\in (0,1]$.
\end{enumerate}
We recall that a Riemannian manifold $(N,g)$ is said to be \emph{parabolic} if there exists a sequence of Lipschitz functions with compact support $\{\phi_k\}_{k\in \mathbb{N}}\subset \mathrm{Lip}_c(N)$ such that $\mathrm{a})$ $0\leq \phi_k\leq 1$, $\mathrm{b})$ $\phi_k\rightarrow 1$ pointwise a.e. as $k\rightarrow \infty$ and $\mathrm{c})$ $\|d_{\min}\phi_k\|_{L^2\Omega^{1}(N,g)}\rightarrow 0$ as $k\rightarrow \infty$. We refer to \cite{BeGu} and the references therein for more on this topic. Moreover we recall that two Riemannian metrics $g_1$ and $g_2$ on a manifold $M$ are said  {\em quasi-isometric} if $c^{-1}g_1\leq g_2\leq cg_1$ for some positive constant $c$.\\
Roughly speaking $g_s$ is a one-parameter family of Hermitian metrics on $M$ that on $A$ degenerates smoothly to a Hermitian pseudometric $h$  for $s\rightarrow 0$ (plus some global control required in the fifth point above). Note that however $g_s$, viewed as a section of $p^*T^*M\otimes p^*T^*M\rightarrow M\times [0,1]$, is allowed to be {\em discontinuous} at $(M\setminus A)\times \{0\}$. In particular $g_s|_{(M\setminus A)}$ might not converge to $h|_{(M\setminus A)}$ as $s\rightarrow 0$. As recalled in the introduction, a Hermitian pseudometric on $M$ is a positive semidefinite Hermitian product on $M$ strictly positive over an open and dense subset. The  {\em degeneracy locus} of $h$ is the smallest closed subset $Z\subset M$ such that $h$ is positive definite over $M\setminus Z$. Obviously $Z\subset M\setminus A$. Clearly $(A,h|_A)$ becomes an incomplete  complex manifold of finite volume. Moreover, as parabolicity is a stable property through quasi-isometries, we known that $(A,g|_A)$ is parabolic with respect to any Riemannian metric $g$ on $M$. In particular $(A,g_s|_A)$ is parabolic for any $s\in (0,1]$. For each $s\in (0,1]$ let us label by 
\begin{equation}
\label{nondeg}
\Delta_{\overline{\partial},m,0,s}:L^2\Omega^{m,0}(M,g_s)\rightarrow L^2\Omega^{m,0}(M,g_s)
\end{equation}
the unique closed (and therefore self-adjoint) extension of $\Delta_{\overline{\partial},m,0,s}:\Omega^{m,0}(M)\rightarrow \Omega^{m,0}(M)$, where the latter operator is the Hodge-Kodaira Laplacian built with respect to the Hermitian metric $g_s$ and acting on the smooth sections of the canonical bundle of $M$.
For $s=0$ let us consider
\begin{equation}
\label{deg}
\Delta_{\overline{\partial},m,0,\abs}:L^2\Omega^{m,0}(A,h|_A)\rightarrow L^2\Omega^{m,0}(A,h|_A)
\end{equation}
which is defined as $\Delta_{\overline{\partial},m,0,\abs}:=\overline{\partial}_{m,0,\max}^*\circ\overline{\partial}_{m,0,\max}$ where 
\begin{equation}
\label{maxx}
\overline{\partial}_{m,0,\max}:L^2\Omega^{m,0}(A,h|_A)\rightarrow L^2\Omega^{m,1}(A,h|_A)
\end{equation}
is the maximal extension of $\overline{\partial}_{m,0}:\Omega_c^{m,0}(A)\rightarrow \Omega_c^{m,1}(A)$, 
\begin{equation}
\label{maxxx}
\overline{\partial}_{m,0,\max}^*:L^2\Omega^{m,1}(A,h|_A)\rightarrow L^2\Omega^{m,0}(A,h|_A)
\end{equation}
is the adjoint of \eqref{maxx} and the domain of $\Delta_{\overline{\partial},m,0,\abs}$ is  $$\mathcal{D}(\Delta_{\overline{\partial},m,0,\abs})=\{\omega\in\mathcal{D}(\overline{\partial}_{m,0,\max})\ \text{such\ that}\ \overline{\partial}_{m,0,\max}\omega \in \mathcal{D}(\overline{\partial}_{m,0,\max}^*)\}.$$
Thanks to \cite{FBei} Th. 4.1 we know that \eqref{deg} has entirely discrete spectrum.
We have now all the ingredients to state the main result of this section.
\begin{teo}
\label{spectralth}
In the setting describe above. Let $0\leq\lambda_1(s)\leq\lambda_2(s)\leq...\leq \lambda_k(s)\leq...$ be the eigenvalues of \eqref{nondeg} and let $0\leq \lambda_1(0)\leq \lambda_2(0)\leq...\leq \lambda_k(0)\leq...$ be the eigenvalues of \eqref{deg}. Then 
\begin{equation}
\label{eigenlim}
\lim_{s\rightarrow 0} \lambda_k(s)=\lambda_k(0) 
\end{equation}
for each positive integer $k$. Moreover let $\{s_n\}$ be any positive sequence such that $s_n\rightarrow 0$ as $n\rightarrow \infty$ and let  $\{\eta_1(s_n),\eta_2(s_n),.$ $..,\eta_k(s_n),...\}$ be  any orthonormal basis of  $L^2\Omega^{m,0}(M,g_{s_n})$  made by eigenforms of \eqref{nondeg} with corresponding eigenvalues  $\{\lambda_1(s_n),...,\lambda_k(s_n),...\}$. Then there exists a subsequence $\{z_n\}\subset \{s_n\}$ and an orthonormal basis $\{\eta_1(0),\eta_2(0),...,\eta_k(0),...\}$ of $L^2\Omega^{m,0}(A,h|_A)$ made by eigenforms of \eqref{deg} with corresponding eigenvalues $\{\lambda_1(0),...,\lambda_k(0),...\}$ such that 
\begin{equation}
\label{eigenflim}
\lim_{n\rightarrow \infty} \eta_k(z_n)=\eta_k(0) 
\end{equation}
in $L^2\Omega^{m,0}(A,h|_A)$ for each positive integer $k$.
\end{teo}
Some remarks to the above statement are in order. More precisely we have to explain why $\eta_k(z_n)\in L^2\Omega^{m,0}(A,h|_A)$ so that the convergence in $L^2\Omega^{m,0}(A,h|_A)$, as required in \eqref{eigenflim}, makes sense. First of all we point out that, as $(A,g|_A)$ is parabolic with respect to any Riemannian metric $g$ on $M$, we can use Th. 3.4 and Prop 3.1 in  \cite{Troya} to conclude that $M\setminus A$ has measure zero. Thus we have an equality of Hilbert spaces $$L^2\Omega^{m,0}(M,g_s)=L^2\Omega^{m,0}(A,g_s|_A)$$ for any $s\in (0,1]$. Moreover, thanks to \eqref{equality}, we know that there is an equality of Hilbert spaces $$L^2\Omega^{m,0}(A,g_s|_A)=L^2\Omega^{m,0}(A,h|_A)$$ for any $s\in [0,1]$. 
Therefore, joining together these equalities, we have 
\begin{equation}
\label{stesso}
L^2\Omega^{m,0}(M,g_s)=L^2\Omega^{m,0}(A,g_{s}|_A)= L^2\Omega^{m,0}(A,h|_A)
\end{equation}
 for any $s\in (0,1]$. Thus \eqref{eigenflim} is well posed. In order to prove Th. \ref{spectralth} we want to apply Th. \ref{mainweapon}. First we need to establish some preliminary properties. Let $F_s$ be a section of $p^*\mathrm{End}(TM)\rightarrow M\times [0,1]$ such that $g_1(F_s\cdot,\cdot)=g_s(\cdot,\cdot)$ for each $s\in (0,1]$ and  $g_1(F_0\cdot,\cdot)=h(\cdot,\cdot)$. Clearly $F_s|_{A\times [0,1]}\in C^{\infty}(A\times [0,1], p^*\mathrm{End}(TA))$, $F_1=\id$ and $F_s$ is  positive definite on $M$ for each fixed $s\in (0,1]$.  We have the following uniform family of continuous inclusions.
 
\begin{prop}
\label{patolax}
The identity $\id:\Omega_c^{m,1}(M)\rightarrow \Omega^{m,1}_c(M)$ induces a continuous inclusion $i:L^2\Omega^{m,1}(M,g_s)\hookrightarrow L^2\Omega^{m,1}(M,g_1)$ such that 
\begin{equation}
\label{fiacco}
\|\omega\|^2_{L^2\Omega^{m,1}(M,g_1)}\leq \nu\|\omega\|^2_{L^2\Omega^{m,1}(M,g_s)}
\end{equation}
for any $s\in (0,1]$ and $\omega\in L^2\Omega^{m,1}(M,g_s)$.
\end{prop}

\begin{proof}
By the assumptions we know that $g_s\leq \nu g_1$ for any $s\in (0,1]$. Therefore, arguing as in \cite[Prop. 1.8]{BePi}, we obtain immediately  that  $\nu g_s^*\geq  g_1^*$ for any $s\in (0,1]$ where $g_s^*$ and $g_1^*$ are the metrics induced by $g_s$ and $g_1$ on $T^*M$, respectively. From the latter inequality we can deduce easily the analogous inequality for the induced Hermitian metrics on $T^{0,1,*}M$, that is $\nu g_{s,0,1}^*\geq  g_{1,0,1}^*$ for any $s\in (0,1]$. As $g_s=g_1(F_s\cdot,\cdot)$ the latter inequality can be reformulated by saying that on $T^{0,1,*}M$ we have  $\nu g_{1,0,1}^*(G_{s,\mathbb{C}}^{0,1}\cdot,\cdot)\geq  g_{1,0,1}^*(\cdot,\cdot)$ for any $s\in (0,1]$.
In this way, given any $\omega\in \Omega_c^{m,1}(M)$ and $s\in (0,1]$,  by \eqref{ibra3} we have 
\begin{align}
& \nonumber \|\omega\|^2_{L^2\Omega^{m,1}(M,g_s)}=\int_Mg_{s,m,1}^*(\omega,\omega)\dvol_{g_s}=\int_Mg_{1,m,1}^*(\det(G^{1,0}_{s,\mathbb{C}})\otimes G_{s,\mathbb{C}}^{0,1}\omega,\omega)\sqrt{\det(F_s)}\dvol_{g_1}=\\
& \nonumber \int_Mg_{1,m,1}^*(\id\otimes G_{s,\mathbb{C}}^{0,1}\omega,\omega)\dvol_{g_1}=\int_Mg_{1,m,0}^*\otimes g^*_{1,0,1}(\id\otimes G_{s,\mathbb{C}}^{0,1}\omega,\omega)\dvol_{g_1}\geq\\
& \int_M\frac{1}{\nu}g_{1,m,1}^*(\omega,\omega)\dvol_{g_1}= \frac{1}{\nu}\|\omega\|^2_{L^2\Omega^{m,1}(M,g_1)}.
\end{align}
In conclusion we have shown that given any $s\in (0,1]$ and $\omega\in \Omega_c^{m,1}(M)$ we have  $$\|\omega\|^2_{L^2\Omega^{m,1}(M,g_1)}\leq \nu\|\omega\|^2_{L^2\Omega^{m,1}(M,g_s)}.$$ Now \eqref{fiacco} follow immediately.
\end{proof}

We have also the following family of uniform continuous inclusions.

\begin{prop}
\label{patola}
The identity $\id:\Omega_c^{m,1}(A)\rightarrow \Omega^{m,1}_c(A)$ induces a continuous inclusion $i:L^2\Omega^{m,1}(A,h|_A)\hookrightarrow L^2\Omega^{m,1}(A,g_s|_A)$ such that for any $\omega\in L^2\Omega^{m,1}(A,h|_A)$ and any $s\in (0,1]$ the following inequality holds true 
\begin{equation}
\label{decia}
\|\omega\|^2_{L^2\Omega^{m,1}(A,g_s|_A)}\leq \nu\|\omega\|^2_{L^2\Omega^{m,1}(A,h|_A)}
\end{equation}
\end{prop}

\begin{proof}
By the assumptions we know that $h\leq \nu g_s$ for any $s\in (0,1]$. Therefore, arguing as in the proof of Prop. \ref{patolax}, we obtain immediately  that over $A$ we have $\nu h^*\geq  g_s^*$ for any $s\in (0,1]$ with $h^*$ and $g_s^*$  the metrics induced by $h$ and $g_s$ on $T^*A$, respectively. As before we get immediately the analogous inequality for the induced Hermitian metrics on $T^{0,1,*}A$, that is $\nu h_{0,1}^*\geq  g_{s,0,1}^*$ for any $s\in (0,1]$. As $g_s(\cdot,\cdot)=g_1(F_s\cdot,\cdot)$ and $h(\cdot,\cdot)=g_1(F_0\cdot,\cdot)$, the latter inequality amounts to saying that on $T^{0,1,*}A$ we have  $\nu g_{1,0,1}^*(G_{0,\mathbb{C}}^{0,1}\cdot,\cdot)\geq  g_{1,0,1}^*(G_{s,\mathbb{C}}^{0,1}\cdot,\cdot)$ for any $s\in (0,1]$. Let now $\omega\in \Omega^{m,1}_c(A)$ and $s\in (0,1]$. Using \eqref{ibra3} we have 
\begin{align}
& \nonumber \|\omega\|_{L^2\Omega^{m,1}(A,g_s|_A)}^2=\int_A g_{s,m,1}^*(\omega,\omega)\dvol_{g_s}=\int_A g_{1,m,1}^*(\det(G_{s,\mathbb{C}}^{1,0})\otimes G_{s,\mathbb{C}}^{0,1}\omega,\omega)\sqrt{\det(F_s)}\dvol_{g_1}=\\
& \nonumber \int_A g_{1,m,1}^*(\id\otimes G_{s,\mathbb{C}}^{0,1}\omega,\omega)\dvol_{g_1}=\int_A g^*_{1,m,0}\otimes g_{1,0,1}^*(\id \otimes G_{s,\mathbb{C}}^{0,1}\omega,\omega)\dvol_{g_1} \leq \nu\int_A g^*_{1,m,0}\otimes g_{1,0,1}^*(\id \otimes G_{0,\mathbb{C}}^{0,1}\omega,\omega)\dvol_{g_1}=\\
& \nonumber \nu\int_A g^*_{1,m,0}\otimes g_{1,0,1}^*(\det(G^{1,0}_{0,\mathbb{C}})\otimes G_{0,\mathbb{C}}^{0,1}\omega,\omega)\sqrt{\det(F_0)}\dvol_{g_1}=\nu\int_A g_{1,m,1}^*(\det(G^{1,0}_{0,\mathbb{C}})\otimes G_{0,\mathbb{C}}^{0,1}\omega,\omega)\sqrt{\det(F_0)}\dvol_{g_1}=\\
& \label{dadaumpa} \nu\int_A h_{m,1}^*(\omega,\omega)\dvol_h=\nu\|\omega\|_{L^2\Omega^{m,1}(A,h|_A)}^2.
\end{align}
In conclusion we showed that for any $s\in (0,1]$ and $\omega\in \Omega_c^{m,1}(A)$ we have 
$$\|\omega\|_{L^2\Omega^{m,1}(A,g_s|_A)}^2\leq \nu\|\omega\|_{L^2\Omega^{m,1}(A,h|_A)}^2$$
as desired.
\end{proof}

\begin{prop}
\label{soothe}
Let $\{s_n\}_{n\in \mathbb{N}}\subset (0,1]$ be any sequence such that $s_n\rightarrow 0$ as $n\rightarrow \infty$. Then:
\begin{enumerate}
\item Consider the sequence of Hilbert spaces $\{L^2\Omega^{m,0}(A,g_{s_n}|_A)\}_{n\in \mathbb{N}}$. Consider $L^2\Omega^{m,0}(A,h|_A)$, let $\mathcal{C}:=L^2\Omega^{m,0}(A,h|_A)$ and for any $n\in \mathbb{N}$, let $\Phi_n:\mathcal{C}\rightarrow L^2\Omega^{m,0}(A,g_{s_n}|_A)$ be the identity map, that is $\Phi_n(\eta)=\eta$, which is well defined thanks to \eqref{equality}. Then $\{L^2\Omega^{m,0}(A,g_{s_n}|_A)\}_{n\in \mathbb{N}}$  converges to $L^2\Omega^{m,0}(A,h|_A)$ in the sense of \eqref{limitHilbert}.
\item Consider the sequence of Hilbert spaces $\{L^2\Omega^{m,1}(A,g_{s_n}|_A)\}_{n\in \mathbb{N}}$. Consider $L^2\Omega^{m,1}(A,h|_A)$, let $\mathcal{C}:=L^2\Omega^{m,1}(A,h|_A)$ and for any $n\in \mathbb{N}$, let $\Phi_n:\mathcal{C}\rightarrow L^2\Omega^{m,1}(A,g_{s_n}|_A)$ be the identity map, that is $\Phi_n(\omega)=\omega$, which is well defined thanks to Prop. \ref{patola}. Then $\{L^2\Omega^{m,1}(A,g_{s_n}|_A)\}_{n\in \mathbb{N}}$  converges to $L^2\Omega^{m,1}(A,h|_A)$ in the sense of \eqref{limitHilbert}.
\end{enumerate}
\end{prop}

\begin{proof}
The first statement is obvious and follows by the fact that we have an equality of Hilbert spaces $L^2\Omega^{m,0}(A,g_{s_n}|_A)=L^2\Omega^{m,0}(A,h|_A)$, for any $n\in \mathbb{N}$, see \eqref{equality}. Now we tackle the second statement. As remarked in the previous proof on $A$ we have $\nu g_{1,0,1}^*(G_{0,\mathbb{C}}^{0,1}\cdot,\cdot)\geq  g_{1,0,1}^*(G_{s,\mathbb{C}}^{0,1}\cdot,\cdot)$ for any $s\in (0,1]$ which clearly in turn implies that $\nu g^*_{1,m,0}\otimes g_{1,0,1}^*(\id\otimes G_{0,\mathbb{C}}^{0,1}\cdot,\cdot)\geq  g^*_{1,m,0}\otimes g_{1,0,1}^*(\id\otimes G_{s,\mathbb{C}}^{0,1}\cdot,\cdot)$ for any $s\in (0,1]$. Let $\omega\in L^2\Omega^{m,1}(A,h|_A)$. Thanks to Prop. \ref{patola} we know that $\omega\in L^2\Omega^{m,1}(A,g_s|_A)$ for any $s\in (0,1]$ and for the corresponding  $L^2$-norm  we have 
\begin{align}
& \label{rogi} \|\omega\|_{L^2\Omega^{m,1}(A,g_s|_A)}^2=\int_A g_{s,m,1}^*(\omega,\omega)\dvol_{g_s}=\int_A g_{1,m,1}^*(\det(G_{s,\mathbb{C}}^{1,0})\otimes G_{s,\mathbb{C}}^{0,1}\omega,\omega)\sqrt{\det(F_s)}\dvol_{g_1}=\\
& \nonumber \int_A g_{1,m,1}^*(\id\otimes G_{s,\mathbb{C}}^{0,1}\omega,\omega)\dvol_{g_1}=\int_A g^*_{1,m,0}\otimes g_{1,0,1}^*(\id \otimes G_{s,\mathbb{C}}^{0,1}\omega,\omega)\dvol_{g_1}.
\end{align}
Moreover we have seen above that  $\nu g^*_{1,m,0}\otimes g_{1,0,1}^*(\id\otimes G_{0,\mathbb{C}}^{0,1}\cdot,\cdot)\geq  g^*_{1,m,0}\otimes g_{1,0,1}^*(\id\otimes G_{s,\mathbb{C}}^{0,1}\cdot,\cdot)$ for any $s\in (0,1]$ and 
$$\int_A g^*_{1,m,0}\otimes g_{1,0,1}^*(\id \otimes G_{0,\mathbb{C}}^{0,1}\omega,\omega)\dvol_{g_1}<\infty$$
as 
$$ \int_A g^*_{1,m,0}\otimes g_{1,0,1}^*(\id \otimes G_{0,\mathbb{C}}^{0,1}\omega,\omega)\dvol_{g_1}=\int_A h^*_{m,1}( \omega,\omega)\dvol_{h}=\|\omega\|_{L^2\Omega^{m,1}(A,h|_A)}^2$$ see \eqref{dadaumpa}. Furthermore it is clear that $g^*_{1,m,0}\otimes g_{1,0,1}^*(\id\otimes G_{s,\mathbb{C}}^{0,1}\omega,\omega)\rightarrow g^*_{1,m,0}\otimes g_{1,0,1}^*(\id\otimes G_{0,\mathbb{C}}^{0,1}\omega,\omega)$ pointwise almost everywhere in $A$ as $s\rightarrow 0$ since $G_{s,\mathbb{C}}^{0,1}\in C^{\infty}(A\times [0,1],p^*\mathrm{End}(T^{0,1,*}A)$.
So we are in position to apply the Lebesgue dominate convergence theorem in \eqref{rogi} and we obtain
\begin{align}
& \nonumber \lim_{s\rightarrow 0} \|\omega\|_{L^2\Omega^{m,1}(A,g_s|_A)}^2=\lim_{s\rightarrow 0}\int_A g_{s,m,1}^*(\omega,\omega)\dvol_{g_s}=\lim_{s\rightarrow 0}\int_A g_{1,m,1}^*(\det(G_{s,\mathbb{C}}^{1,0})\otimes G_{s,\mathbb{C}}^{0,1}\omega,\omega)\sqrt{\det(F_s)}\dvol_{g_1}=\\
& \nonumber \int_A \lim_{s\rightarrow 0}g^*_{1,m,0}\otimes g_{1,0,1}^*(\id \otimes G_{s,\mathbb{C}}^{0,1}\omega,\omega)\dvol_{g_1}=\int_A g^*_{1,m,0}\otimes g_{1,0,1}^*(\id \otimes G_{0,\mathbb{C}}^{0,1}\omega,\omega)\dvol_{g_1}=\|\omega\|_{L^2\Omega^{m,1}(A,h|_A)}^2.
\end{align}
In conclusion we have shown that for any $\omega\in L^2\Omega^{m,1}(A,h|_A)$ the following equality holds true:  $$\lim_{s\rightarrow 0}\langle\omega,\omega\rangle_{L^2\Omega^{m,1}(A,g_{s}|_A)}=\langle\omega,\omega\rangle_{L^2\Omega^{m,1}(A,h|_A)}.$$ This completes the proof as the second statement of this proposition is a straightforward consequence of the above equality.
\end{proof}

We have the following immediate consequence:

\begin{cor}
\label{swim}
 Let $\omega\in L^2\Omega^{m,1}(A,h|_A)$ be arbitrarily fixed. Then the constant sequence $\{\omega_n\}_{n\in \mathbb{N}}$, $\omega_n:=\omega$, viewed as a sequence where $\omega_n\in L^2\Omega^{m,1}(A,g_{s_n}|_A)$ for any $n\in \mathbb{N}$, converges strongly in the sense of Def. \ref{strong} to $\omega$ as $n\rightarrow \infty$.
\end{cor}

Now, for each $s\in (0,1]$, consider  the operators 
\begin{align}
\nonumber & \overline{\partial}_{m,0,\max}:L^2\Omega^{m,0}(A,g_s|_A)\rightarrow L^2\Omega^{m,1}(A,g_s|_A)\\
\nonumber & \overline{\partial}_{m,0,\min}:L^2\Omega^{m,0}(A,g_s|_A)\rightarrow L^2\Omega^{m,1}(A,g_s|_A)\\
\label{boxuomo} & \overline{\partial}_{m,0}:L^2\Omega^{m,0}(M,g_s)\rightarrow L^2\Omega^{m,1}(M,g_s)
\end{align}
where the first two are the maximal/minimal extensions of $\overline{\partial}_{m,0}:\Omega_c^{m,0}(A)\rightarrow \Omega_c^{m,1}(A)$ and the third one is the unique $L^2$ closed extension of $\overline{\partial}_{m,0}:\Omega^{m,0}(M) \rightarrow \Omega^{m,1}(M)$. As showed in \cite{FBei} Prop. 3.2 the above three operators coincide.  In particular  $\overline{\partial}_{m,0}:\Omega_c^{m,0}(A)\rightarrow \Omega_c^{m,1}(A)$ has a unique closed extension, that we label with 
\begin{equation}
\label{urcau}
\overline{\partial}_{m,0}:L^2\Omega^{m,0}(A,g_s|_A)\rightarrow L^2\Omega^{m,1}(A,g_s|_A)
\end{equation}
 and that coincides with \eqref{boxuomo}. Let us consider now $\overline{\partial}^{t,s}_{m,0}:\Omega^{m,1}_c(A)\rightarrow \Omega^{m,0}_c(A)$, that is the formal adjoint of $\overline{\partial}_{m,0}:\Omega^{m,0}_c(A)\rightarrow \Omega^{m,1}_c(A)$ with respect to $g_s$. Analogously to the previous case also the operators
\begin{align}
\nonumber & \overline{\partial}_{m,0,\max}^{t,s}:L^2\Omega^{m,1}(A,g_s|_A)\rightarrow L^2\Omega^{m,0}(A,g_s|_A)\\
\nonumber & \overline{\partial}_{m,0,\min}^{t,s}:L^2\Omega^{m,1}(A,g_s|_A)\rightarrow L^2\Omega^{m,0}(A,g_s|_A)\\
\label{boxdonna} & \overline{\partial}_{m,0}^{*,s}:L^2\Omega^{m,1}(M,g_s)\rightarrow L^2\Omega^{m,0}(M,g_s)
\end{align}
where the first two are the maximal/minimal extensions of $\overline{\partial}_{m,0}^{t,s}:\Omega_c^{m,1}(A)\rightarrow \Omega_c^{m,0}(A)$ and the third one is the unique $L^2$ closed extension of $\overline{\partial}_{m,0}^{t,s}:\Omega^{m,1}(M)\rightarrow \Omega^{m,0}(M)$,  coincide. Therefore $\overline{\partial}_{m,0}^{t,s}:\Omega_c^{m,1}(A)\rightarrow \Omega_c^{m,0}(A)$ has a unique closed extension,  denoted by 
\begin{equation}
\label{uzz}
\overline{\partial}_{m,0}^{t,s}:L^2\Omega^{m,1}(A,g_s|_A)\rightarrow L^2\Omega^{m,0}(A,g_s|_A)
\end{equation}
that coincides with \eqref{boxdonna}.
This allows us to conclude that  the operator \eqref{nondeg} coincides with  
\begin{equation}
\label{kelvin}
\overline{\partial}_{m,0}^{t,s}\circ \overline{\partial}_{m,0}:L^2\Omega^{m,0}(A,g_s|_A)\rightarrow L^2\Omega^{m,0}(A,g_s|_A)
\end{equation}
where $\overline{\partial}_{m,0}:L^2\Omega^{m,0}(A,g_s|_A)\rightarrow L^2\Omega^{m,1}(A,g_s|_A)$ is defined in \eqref{urcau}, $\overline{\partial}_{m,0}^{t,s}:L^2\Omega^{m,1}(A,g_s|_A)\rightarrow L^2\Omega^{m,0}(A,g_s|_A)$
is defined in \eqref{uzz} and the domain of \eqref{kelvin} is $$\mathcal{D}(\overline{\partial}_{m,0}^{t,s}\circ \overline{\partial}_{m,0})=\{\omega \in \mathcal{D}(\overline{\partial}_{m,0})\ \text{such that}\ \overline{\partial}_{m,0}\omega\in \mathcal{D}(\overline{\partial}_{m,0}^{t,s})\}.$$

Using the above remarks, Prop. \ref{patola} and \eqref{equality} it is not hard to show the next property. For a complete proof we refer to \cite{FBei} Th. 4.1. 

\begin{prop}
\label{extension}
Let $\mathcal{D}(\overline{\partial}_{m,0,\max})$ be the domain of \eqref{maxx}. For each $s\in (0,1]$ let $\mathcal{D}(\overline{\partial}_{m,0})$ be the domain of \eqref{urcau}. Then we have a continuous inclusion $\mathcal{D}(\overline{\partial}_{m,0,\max})\hookrightarrow \mathcal{D}(\overline{\partial}_{m,0})$ where each domain is endowed with the corresponding graph norm. Moreover for any $\omega\in \mathcal{D}(\overline{\partial}_{m,0,\max})$ we have $\overline{\partial}_{m,0}\omega=\overline{\partial}_{m,0,\max}\omega$.
\end{prop}

\begin{rem}
The reader might wonder why we did not denote by $\overline{\partial}_{m,0}^s$ the operator \eqref{urcau} in order to emphasize explicitly the dependence on $s$. The reason is that the operator \eqref{urcau} does not depend on $s$. Indeed $\overline{\partial}_{m,0}:\Omega_c^{m,0}(A)\rightarrow \Omega_c^{m,1}(A)$ is an intrinsic operator that does not depend on the metric. If we now consider its closure with respect to $g_s$ then, by  the fact that for any $0<s_1\leq s_2\leq 1$ the metrics $g_{s_1}$ and $g_{s_2}$ are quasi-isometric, we can deduce easily that a $(m,0)$-form $\omega\in L^2\Omega^{m,0}(A,g_{s_1})=L^2\Omega^{m,0}(A,g_{s_2})$ lies in the domain of the  unique closure of $\overline{\partial}_{m,0}:\Omega_c^{m,0}(A)\rightarrow \Omega_c^{m,1}(A)$ with respect to $g_{s_1}$ if and only if it lies in the domain of the unique closure of $\overline{\partial}_{m,0}:\Omega_c^{m,0}(A)\rightarrow \Omega_c^{m,1}(A)$ with respect to $g_{s_2}$ and  the action of $\overline{\partial}_{m,0}$ on $\omega$ with respect to $g_{s_1}$ coincides with the action of $\overline{\partial}_{m,0}$ on $\omega$ with respect to $g_{s_2}$. Thus, as long as $s\in (0,1]$, the operator \eqref{urcau} is uniquely determined.
\end{rem}

We have all the ingredients to introduce the family of quadratic forms we will work with. Let $\{s_n\}_{n\in \mathbb{N}}\subset (0,1]$ be a sequence with $s_n\rightarrow 0$ as $n\rightarrow \infty$. We  define 
\begin{equation}
\label{amaro}
\mathcal{D}(Q_{s_n}):=\mathcal{D}(\overline{\partial}_{m,0})\quad\quad \text{and}\quad\quad Q_{s_n}(\omega,\eta):=\langle\overline{\partial}_{m,0}\omega,\overline{\partial}_{m,0}\eta\rangle_{L^2\Omega^{m,1}(A,g_{s_n}|_A)}
\end{equation}
for any $\omega,\eta\in \mathcal{D}(Q_{s_n})$, where $\overline{\partial}_{m,0}:L^2\Omega^{m,0}(A,g_{s_n}|_A)\rightarrow L^2\Omega^{m,1}(A,g_{s_n}|_A)$ is defined in 
\eqref{urcau}. Clearly $Q_{s_n}$  is a closed quadratic form, that is $(\mathcal{D}(Q_{s_n}),Q_{s_n,H})$ is a Hilbert space. It is clear that the latter space is a Hilbert space as it is nothing but the domain of \eqref{urcau} endowed with its graph product. We remind that $Q_{s_n,H}(\cdot,\cdot):=\langle\cdot,\cdot\rangle_{L^2\Omega^{m,0}(A,g_{s_n}|_A)}+Q_{s_n}(\cdot,\cdot)$. Summarizing $Q_{s_n}$ is the closed quadratic form associated to the operator \eqref{nondeg}.
Finally we introduce the quadratic form  $Q_0$ defined as 
\begin{equation}
\label{amarone}
\mathcal{D}(Q_0):=\mathcal{D}(\overline{\partial}_{m,0,\max})\quad\quad \text{and}\quad\quad Q_0(\omega,\eta):=\langle\overline{\partial}_{m,0,\max}\omega,\overline{\partial}_{m,0,\max}\eta\rangle_{L^2\Omega^{m,1}(A,h|_A)}
\end{equation}
for any $\omega,\eta\in \mathcal{D}(Q_0)$, where $\overline{\partial}_{m,0,\max}:L^2\Omega^{m,0}(A,h|_A)\rightarrow L^2\Omega^{m,1}(A,h|_A)$ is defined in  \eqref{maxx}. In other words $Q_0$ is the quadratic form associated to the operator \eqref{deg} and, likewise the previous case, it is a closed quadratic form, that is 
$(\mathcal{D}(Q_{0}),Q_{0,H})$ is a Hilbert space. In the next propositions we show various properties concerning $\{(\mathcal{D}(Q_{s_n}),Q_{s_n,H})\}_{n\in \mathbb{N}}$ and $(\mathcal{D}(Q_{0}),Q_{0,H})$. With $\{s_n\}_{n\in \mathbb{N}}$ we  denote any sequence with $\{s_n\}\subset (0,1]$ such that  $s_n\rightarrow 0$ as $n\rightarrow \infty$.

\begin{prop}
\label{pakistan}
We have the following properties:
\begin{enumerate}
\item Let $\omega\in \mathcal{D}(Q_{0})$. Then for any $n\in \mathbb{N}$ $\omega\in \mathcal{D}(Q_{s_n})$ and the corresponding inclusion $i_{0,n}:(\mathcal{D}(Q_{0}),Q_{0,H})\hookrightarrow (\mathcal{D}(Q_{s_n}),Q_{s_n,H})$ is continuous. More precisely, defining $\tau:=\max\{1,\nu\}$, we have $$\|\omega\|^2_{L^2\Omega^{m,0}(A,g_{s_n}|_A)}+Q_{s_n}(\omega,\omega)\leq \tau\left(\|\omega\|^2_{L^2\Omega^{m,0}(A,h|_A)}+Q_{0}(\omega,\omega)\right)$$ for any $s_n$ and $\omega\in \mathcal{D}(Q_{0})$.
\item Let $\mathcal{C}:=\mathcal{D}(Q_{0})$ and $\Phi_n:=i_{0,n}$, that is the inclusion defined in the previous point. Then $$\{(\mathcal{D}(Q_{s_n}),Q_{s_n,H})\}_{n\in \mathbb{N}}$$ converges to $(\mathcal{D}(Q_{0}),Q_{0,H})$ as $n\rightarrow \infty$.
\item Let $n\in \mathbb{N}$ and let $\omega\in \mathcal{D}(Q_{s_n})$. Then $\omega\in \mathcal{D}(Q_{s_1})$ and the corresponding inclusion $i_{n,1}:(\mathcal{D}(Q_{s_n}),Q_{s_n,H})\hookrightarrow (\mathcal{D}(Q_{1}),Q_{1,H})$ is continuous. More precisely, with $\tau$ defined as above, we have $$\|\omega\|^2_{L^2\Omega^{m,0}(A,g_1|_A)}+Q_1(\omega,\omega)\leq \tau\left(\|\omega\|^2_{L^2\Omega^{m,0}(A,g_{s_n}|_A)}+Q_{s_n}(\omega,\omega)\right)$$ for any  $s_n$ and  $\omega\in \mathcal{D}(Q_{s_n})$.
\end{enumerate}
\end{prop}

\begin{proof}
The first point follows immediately by \eqref{equality}, Prop. \ref{patola} and Prop. \ref{extension}. The second point is a straightforward application of the first point and  Prop. \ref{soothe}. Concerning the third point we first note that $\omega\in \mathcal{D}(Q_{s_1})$ as $g_{s_n}$ and $g_{s_1}$ are quasi-isometric. Now the rest of the proof follows immediately  by \eqref{equality} and Prop. \ref{patolax}.
\end{proof}

\begin{prop}
\label{www}
Let $\phi\in \Omega_c^{m,1}(A)$ be an arbitrarily fixed $(m,1)$-form with compact support. Then $\overline{\partial}^{t,s_n}_{m,0}\phi\rightarrow \overline{\partial}^{t,0}_{m,0}\phi$ weakly in $L^2\Omega^{m,0}(A,h|_A)$ as $n\rightarrow \infty$, where $\overline{\partial}_{m,0}^{t,0}:\Omega^{m,1}_c(A)\rightarrow \Omega_c^{m,0}(A)$ is the formal adjoint of $\overline{\partial}_{m,0}:\Omega^{m,0}_c(A)\rightarrow \Omega_c^{m,1}(A)$ with respect to $h|_A$.
\end{prop}
\begin{proof}
Let $\chi:A\times [0,1]\rightarrow \mathbb{R}$ be the function defined as $\chi(p,s):=h^*_{m,0}(\overline{\partial}^{t,s}_{m,0}\phi,\overline{\partial}_{m,0}^{t,s}\phi)$, that is for any $(p,s)\in A\times [0,1]$, $\chi(p,s)$ is given by the square of the pointwise norm in $p$ of $\overline{\partial}^{t,s}_{m,0}\phi$ with respect to $h^*_{m,0}$. By the fact that $g_s\in C^{\infty}(A\times [0,1], p^*T^*A\otimes p^*T^*A)$ and $\phi\in \Omega^{m,1}_c(A)$ we know that $\chi$ is continuous on $A\times [0,1]$ and $\supp(\chi)\subset \supp(\phi)\times [0,1]$. In particular $\supp(\chi)$ is a compact subset of $A\times [0,1]$. Therefore there exists a positive constant $b\in \mathbb{R}$ such that $\chi(p,s)\leq b$ for any $p\in A$ and $s\in [0,1]$, that is $h^*_{m,0}(\overline{\partial}_{m,0}^{t,s}\phi,\overline{\partial}_{m,0}^{t,s}\phi)\leq b$ on $A\times [0,1]$. This latter inequality tells us that $\|\overline{\partial}_{m,0}^{t,s}\phi\|^2_{L^2\Omega^{m,0}(A,h|_A)}\leq b\vol_{h}(A)$ for any $s\in[0,1]$. Now, as we know that  $\{\|\overline{\partial}_{m,0}^{t,s_n}\phi\|_{L^2\Omega^{m,0}(A,h|_A)}\}_{n\in \mathbb{N}}$ is a bounded sequence, in order to conclude the proof it is enough to fix a dense subset $Z$ of $L^2\Omega^{m,0}(A,h|_A)$ and to show that 
$$\lim_{n\rightarrow \infty}\langle\omega,\overline{\partial}_{m,0}^{t,s_n}\phi\rangle_{L^2\Omega^{m,0}(A,h|_A)}=\langle\omega,\overline{\partial}_{m,0}^{t,0}\phi\rangle_{L^2\Omega^{m,0}(A,h|_A)}$$ for any $\omega\in Z$. Let us fix $Z:=\mathcal{D}(Q_0)$ and let $\omega\in \mathcal{D}(Q_0)$. Then, using \eqref{equality} and Prop. \ref{extension}, we have $$\langle\omega,\overline{\partial}_{m,0}^{t,s_n}\phi\rangle_{L^2\Omega^{m,0}(A,h|_A)}=\langle\omega,\overline{\partial}_{m,0}^{t,s_n}\phi\rangle_{L^2\Omega^{m,0}(A,g_{s_n}|_A)}=\langle\overline{\partial}_{m,0}\omega,\phi\rangle_{L^2\Omega^{m,1}(A,g_{s_n}|_A)}.$$ In this way, keeping in mind \eqref{equality}, Cor. \ref{swim} and Prop. \ref{extension}, we have 
\begin{align}
\nonumber & \lim_{n\rightarrow \infty}\langle\omega,\overline{\partial}_{m,0}^{t,s_n}\phi\rangle_{L^2\Omega^{m,0}(A,h|_A)}=\lim_{n\rightarrow \infty}\langle\omega,\overline{\partial}_{m,0}^{t,s_n}\phi\rangle_{L^2\Omega^{m,0}(A,g_{s_n}|_A)}=\lim_{n\rightarrow \infty}\langle\overline{\partial}_{m,0}\omega,\phi\rangle_{L^2\Omega^{m,1}(A,g_{s_n}|_A)}=\\
& \nonumber \lim_{n\rightarrow \infty}\langle\overline{\partial}_{m,0,\max}\omega,\phi\rangle_{L^2\Omega^{m,1}(A,g_{s_n}|_A)}= \langle\overline{\partial}_{m,0,\max}\omega,\phi\rangle_{L^2\Omega^{m,1}(A,h|_A)}=\langle\omega,\overline{\partial}_{m,0}^{t,0}\phi\rangle_{L^2\Omega^{m,0}(A,h|_A)}
\end{align}
as desired.
\end{proof}

\begin{prop}
\label{asycom}
The sequence of closed quadratic forms $\{Q_{s_n}\}_{n\in \mathbb{N}}$ is asymptotically compact.
\end{prop}

\begin{proof}
Let $\{\omega_n\}_{n\in \mathbb{N}}$ be a sequence with $\omega_n\in L^2\Omega^{m,0}(A,g_{s_n}|_A)$ such that $$\limsup_{n\rightarrow \infty}\left(\|\omega_n\|_{L^2\Omega^{m,0}(A,g_{s_n}|_A)}+\overline{Q}_{s_n}(\omega_n,\omega_n)\right)<\infty.$$
We can deduce the existence of a positive constant $c$ and a subsequence $\{\omega_{\ell}\}_{\ell\in \mathbb{N}}\subset \{\omega_n\}_{n\in \mathbb{N}}$ such that $\omega_{\ell}\in \mathcal{D}(Q_{s_{\ell}})$  and $$\|\omega_{\ell}\|_{L^2\Omega^{m,0}(A,g_{s_{\ell}}|_A)}+Q_{s_{\ell}}(\omega_{\ell},\omega_{\ell})\leq c$$ for any $\ell\in \mathbb{N}$. Hence, thanks to Prop. \ref{pakistan}, we know that $\{\omega_{\ell}\}_{\ell\in \mathbb{N}}\subset \mathcal{D}(Q_1)$ and $$\|\omega_{\ell}\|_{L^2\Omega^{m,0}(A,g_1|_A)}+Q_1(\omega_{\ell},\omega_{\ell})\leq \tau c$$ where $\tau$ is defined in Prop. \ref{pakistan}.
As the injection $(\mathcal{D}(Q_1), Q_{1,H})\hookrightarrow L^2\Omega^{m,0}(A,g_1|_A)$ is a compact operator, see \cite{FBei} pag. 774, we can conclude that there exists a subsequence $\{\omega_v\}_{v\in \mathbb{N}}\subset \{\omega_{\ell}\}_{\ell\in \mathbb{N}}$ and an element $\omega\in L^2\Omega^{m,0}(A,g_1|_A)$ such that $\omega_v\rightarrow \omega$ in $L^2\Omega^{m,0}(A,g_1|_A)$ as $v\rightarrow \infty$. Now looking at the sequence $\{\omega_v\}_{v\in \mathbb{N}}$ as a sequence where each element $\omega_v\in L^2\Omega^{m,0}(A,g_{s_v}|_A)$ and keeping in mind \eqref{equality}, it is immediate to check $\omega_v\rightarrow \omega \in L^2\Omega^{m,0}(A,h|_A)$ strongly as $v\rightarrow \infty$ in the sense of Def. \ref{strong}. Indeed, by Prop. \ref{soothe}, we have $\mathcal{C}=L^2\Omega^{m,0}(A,h|_A)$ and $\Phi_v:\mathcal{C}\rightarrow L^2\Omega^{m,0}(A,g_{s_v}|_A)$ is just the identity $\id:L^2\Omega^{m,0}(A,h|_A)\rightarrow L^2\Omega^{m,0}(A,g_{s_v}|_A)$. Therefore, taking the constant sequence $\{\omega\}$ as a net in $\mathcal{C}$ converging to $\omega$, \eqref{stronglimit} becomes 
\begin{equation}
\label{settimo}
\limsup_{v\rightarrow \infty}\|\omega-\omega_{v}\|_{L^2\Omega^{m,0}(A,h|_A)}
\end{equation}
 and we have already shown above that \eqref{settimo} is zero. The proposition is thus established.
\end{proof}

\begin{prop}
\label{domain}
 Let $\omega \in L^{2}\Omega^{m,0}(A,h|_A)$. Assume that there exists a sequence $\{\omega_n\}_{n\in \mathbb{N}}$ such that $\omega_n\in \mathcal{D}(Q_{s_n})$, $\omega_n\rightarrow \omega$ weakly as $n\rightarrow \infty$ and $Q_{s_n}(\omega_n,\omega_n)\leq c$ for any $n\in \mathbb{N}$ and a positive constant $c$. Then $\omega\in \mathcal{D}(Q_0)$.
\end{prop}

\begin{proof}
By the hypothesis and the very definition of $Q_{s_n}$ we have  $\langle \overline{\partial}_{m,0} \omega_n,\overline{\partial}_{m,0} \omega_n\rangle_{L^2\Omega^{m,1}(A,g_{s_n})}\leq c$ for any $n\in \mathbb{N}$. Thanks to Prop. \ref{bounded} we know that there exists $\eta\in L^2\Omega^{m,1}(A,h|_A)$ and a subsequence $\{\omega_v\}_{v\in \mathbb{N}}\subset \{\omega_n\}_{n\in \mathbb{N}}$, $\omega_v\in \mathcal{D}(Q_{s_v})$ such that $\overline{\partial}_{m,0}\omega_v\rightarrow \eta$ weakly in the sense of Def. \ref{weak} as $v\rightarrow \infty$. Consider now the subsequence $\{\omega_v\}_{v\in \mathbb{N}}$. Since $\omega_v\rightarrow \omega$ weakly as $v\rightarrow \infty$ we have $\|\omega_v\|_{L^2\Omega^{m,0}(A,h|_A)}\leq d$ for some positive constant $d$ and any $v\in \mathbb{N}$, see for instance Prop. \ref{wibounded}. Hence $$Q_{s_v,H}(\omega_v,\omega_v)\leq d+c.$$ Thus by Prop. \ref{asycom} we know that there is a subsequence $\{\omega_{w}\}_{w\in \mathbb{N}}\subset \{\omega_{v}\}_{v\in \mathbb{N}}$ such that $\omega_w\rightarrow \omega$ in $L^2\Omega^{m,0}(A,h|_A)$ as $w\rightarrow \infty$.
We are in position to check that $\omega \in \mathcal{D}(Q_{0})$. Let $\phi\in \Omega_c^{m,1}(A)$. Thanks to Prop. \ref{www} we have

\begin{align}
\nonumber & \langle\omega,\overline{\partial}_{m,0}^{t,0}\phi\rangle_{L^2\Omega^{m,0}(A,h|_A)}= \lim_{w\rightarrow \infty} \langle\omega,\overline{\partial}_{m,0}^{t,s_w}\phi\rangle_{L^2\Omega^{m,0}(A,g_{s_w}|_A)}=\\
&\nonumber \lim_{w\rightarrow \infty}\langle\omega_w,\overline{\partial}_{m,0}^{t,s_w}\phi\rangle_{L^2\Omega^{m,0}(A,g_{s_w}|_A)}= \lim_{w\rightarrow \infty}\langle\overline{\partial}_{m,0}\omega_w,\phi\rangle_{L^2\Omega^{m,1}(A,g_{s_w}|_A)}=\langle \eta, \phi \rangle_{L^2\Omega^{m,1}(A,h|_A)}.
\end{align}

Thus we proved $$\langle\omega,\overline{\partial}_{m,0}^{t,0}\phi\rangle_{L^2\Omega^{m,0}(A,h|_A)}=\langle \eta, \phi \rangle_{L^2\Omega^{m,1}(A,h|_A)}$$ for any $\phi\in \Omega_c^{m,1}(A)$ and so $\omega\in \mathcal{D}(Q_0)$ and $\overline{\partial}_{m,0,\max}\omega=\eta$ as desired.
\end{proof}

\begin{prop}
\label{maining}
The sequence of closed quadratic forms $\{Q_{s_n}\}_{n\in \mathbb{N}}$ Mosco-converges to $Q_0$.
\end{prop}

\begin{proof}
According to Def. \ref{Mosco} we divide the proof in two steps. First we want to show that  
\begin{itemize}
\item for any sequence $\{\omega_n\}_{n\in \mathbb{N}}\subset L^2\Omega^{m,0}(A,h|_A)$  weakly convergent to some $\omega\in L^2\Omega^{m,0}(A,h|_A)$ we have 
\begin{equation}
\label{aranw}
\overline{Q}_0(\omega)\leq \liminf_{n\rightarrow \infty}\overline{Q}_{s_n}(\omega_n,\omega_n).
\end{equation}
\end{itemize}
Let's consider first the case $\omega\in \mathcal{D}(Q_0)$. If $\liminf_{n\rightarrow \infty}\overline{Q}_{s_n}(\omega_n,\omega_n)=\infty$ then the above inequality is clearly fulfilled. Assume now that $\liminf_{n\rightarrow \infty}\overline{Q}_{s_n}(\omega_n,\omega_n)<\infty$. Then we can extract a subsequence $\{\omega_u\}_{u\in \mathbb{N}}\subset \{\omega_n\}_{n\in \mathbb{N}}$ such that $\omega_u\in \mathcal{D}(Q_{s_u})$ and 
\begin{equation}
\label{draw}
\lim_{u\rightarrow \infty}Q_{s_u}(\omega_u,\omega_u)=\liminf_{n\rightarrow \infty}\overline{Q}_{s_n}(\omega_n,\omega_n)=c
\end{equation}
 for some $c\in \mathbb{R}$. Thus, thanks to Prop. \ref{bounded}, we can pass to a new subsequence $\{\omega_v\}_{v\in \mathbb{N}}\subset \{\omega_u\}_{u\in \mathbb{N}}$, $\omega\in \mathcal{D}(Q_{s_v})$, such that $\overline{\partial}_{m,0}\omega_v\rightarrow \eta$ weakly in the sense of Def. \ref{weak} to some $\eta\in L^2\Omega^{m,1}(A,h|_A)$. As $\omega_v\rightarrow \omega$ weakly in $L^2\Omega^{m,0}(A,h|_A)$ as $v\rightarrow \infty$ we know that $\{\omega_v\}_{v\in \mathbb{N}}$ is a bounded sequence. Therefore $$\limsup_{v\rightarrow\infty}Q_{s_v,H}(\omega_v,\omega_v)<\infty$$ and so, thanks to Prop. \ref{asycom}, we know that there is a subsequence $\{\omega_w\}_{w\in \mathbb{N}}\subset \{\omega_v\}_{v\in \mathbb{N}}$, $\omega_w\in \mathcal{D}(Q_{s_w})$,  such that $\omega_w\rightarrow \omega$ in $L^2\Omega^{m,0}(A,h|_A)$ as $w\rightarrow \infty$. Now we claim that $\omega\in \mathcal{D}(\overline{\partial}_{m,0,\max})$ and  $$\overline{\partial}_{m,0,\max}\omega=\eta.$$ This follows by arguing as in the proof of Prop. \ref{domain}. Indeed let $\phi\in \Omega_c^{m,1}(A)$. By Prop. \ref{www} we have 
\begin{align}
\nonumber & \langle\omega,\overline{\partial}_{m,0}^{t,0}\phi\rangle_{L^2\Omega^{m,0}(A,h|_A)}= \lim_{w\rightarrow \infty} \langle\omega,\overline{\partial}_{m,0}^{t,s_w}\phi\rangle_{L^2\Omega^{m,0}(A,g_{s_w}|_A)}=\\
&\nonumber \lim_{w\rightarrow \infty}\langle\omega_w,\overline{\partial}_{m,0}^{t,s_w}\phi\rangle_{L^2\Omega^{m,0}(A,g_{s_w}|_A)}= \lim_{w\rightarrow \infty}\langle\overline{\partial}_{m,0}\omega_w,\phi\rangle_{L^2\Omega^{m,1}(A,g_{s_w}|_A)}=\langle \eta, \phi \rangle_{L^2\Omega^{m,1}(A,h|_A)}.
\end{align}

Finally, thanks to \eqref{lerume} and \eqref{draw}, we have 
\begin{align}
& \nonumber Q_{0}(\omega,\omega)=\langle\overline{\partial}_{m,0,\max}\omega,\overline{\partial}_{m,0,\max}\omega   \rangle_{L^2\Omega^{m,1}(A,h|_A)}=\langle\eta,\eta\rangle_{L^2\Omega^{m,1}(A,h|_A)}\leq\\
& \nonumber  \liminf_{w\rightarrow \infty} \langle \overline{\partial}_{m,0}\omega_w,\overline{\partial}_{m,0}\omega_w \rangle_{L^2\Omega^{m,1}(A,g_{s_w}|_A)}=\lim_{w\rightarrow \infty} \langle \overline{\partial}_{m,0}\omega_w,\overline{\partial}_{m,0}\omega_w \rangle_{L^2\Omega^{m,1}(A,g_{s_w}|_A)}=\\
&\nonumber \lim_{w\rightarrow \infty} Q_{s_w}(\omega_w,\omega_w)= \liminf_{n\rightarrow \infty} \overline{Q}_{s_n}(\omega_n,\omega_n).
\end{align}
The above inequality $\langle\eta,\eta\rangle_{L^2\Omega^{m,1}(A,h|_A)}\leq$  $\liminf_{w\rightarrow \infty} \langle \overline{\partial}_{m,0}\omega_w,\overline{\partial}_{m,0}\omega_w \rangle_{L^2\Omega^{m,1}(A,g_{s_w}|_A)}$ follows by \eqref{lerume} combined with the fact that $\{\omega_w\}\subset\{\omega_v\}$ and $\overline{\partial}_{m,0}\omega_v\rightarrow \eta$ weakly in the sense of Def. \ref{weak} to  $\eta\in L^2\Omega^{m,1}(A,h|_A)$. Moreover $\lim_{w\rightarrow \infty} \langle \overline{\partial}_{m,0}\omega_w,\overline{\partial}_{m,0}\omega_w \rangle_{L^2\Omega^{m,1}(A,g_{s_w}|_A)}$ exists thanks to \eqref{draw} and the fact that $\{\omega_w\}_{w\in\mathbb{N}}\subset \{\omega_u\}_{u\in\mathbb{N}}$.
In conclusion we proved that $$Q_0(\omega,\omega)\leq \liminf_{n\rightarrow\infty}{\overline{Q}}_{s_n}(\omega_n,\omega_n)$$ as desired. Now consider the case $\omega\notin \mathcal{D}(Q_0)$. Then, thanks to Prop. \ref{domain}, we have $$\liminf_{n\rightarrow\infty}\overline{Q}_{s_n}(\omega_n,\omega_n)=\infty$$ for any sequence $\{\omega_n\}_{n\in\mathbb{N}}$ weakly convergent to $\omega$ in $L^2\Omega^{m,0}(A,h|_A)$. In particular \eqref{aranw} is satisfied. This establishes the first part of the proof. Now we come to the second part. We have to show that:
\begin{itemize}
\item  for any $\omega\in L^2\Omega^{m,0}(A,h|_A)$ there exists a sequence $\{\omega_n\}_{n\in \mathbb{N}}\subset L^2\Omega^{m,0}(A,h|_A)$ with $\omega_n\rightarrow \omega$ in $L^2\Omega^{m,0}(A,h|_A)$ such that $$\overline{Q}_0(\omega,\omega)=\lim_{n\rightarrow \infty} \overline{Q}_{s_n}(\omega_n,\omega_n).$$
\end{itemize}
Let $\omega\in L^2\Omega^{m,0}(A,h|_A)$. Consider the constant sequence $\{\omega_n\}_{n\in \mathbb{N}}$, $\omega_n:=\omega$, which clearly converges to $\omega$ in the sense of Def. \ref{strong}. Let's consider first the case $\omega\in \mathcal{D}(Q_0)$. Then, thanks to Prop. \ref{pakistan}, $\omega\in \mathcal{D}(Q_{s_n})$ and $Q_{s_n}(\omega,\omega)\rightarrow Q_0(\omega,\omega)$ as $n\rightarrow \infty$. Assume now that  $\omega\notin \mathcal{D}(Q_0)$. If there exists a positive integer $\overline{n}$ such that $\omega\notin \mathcal{D}(Q_{s_n})$ for $n\geq \overline{n}$ then we have $\overline{Q}_0(\omega,\omega)=\infty=\overline{Q}_{s_n}(\omega,\omega)$ for any $n\geq \overline{n}$ and therefore $\overline{Q}_0(\omega,\omega)=\lim \overline{Q}_{s_n}(\omega,\omega)$ as $n\rightarrow \infty$. Finally let us tackle the remaining case: $\omega \notin \mathcal{D}(Q_0)$ and $\omega\in \mathcal{D}(Q_{s_n})$ for each $n\in \mathbb{N}$. Then, by Prop. \ref{domain}, we have $\liminf_{n\rightarrow \infty}Q_{s_n}(\omega_n,\omega_n)=\infty$. Therefore $\lim_{n\rightarrow \infty}Q_{s_n}(\omega_n,\omega_n)=\infty=\overline{Q}_0(\omega,\omega)$ and this concludes the proof.
\end{proof}

Now we can conclude the proof of Th. \ref{spectralth}.
\begin{proof}
Let $\{s_n\}_{n\in \mathbb{N}}\subset (0,1]$ be any sequence with $s_n\rightarrow \infty$ as $n\rightarrow \infty$. We have seen, thanks to Prop. \ref{asycom} and Prop. \ref{maining}, that $\{Q_{s_n}\}_{n\in \mathbb{N}}$ compactly converges to $Q_{0}$ as $n\rightarrow \infty$. As $Q_{s_n}$ is the closed quadratic form associated to \eqref{nondeg} and $Q_0$ is the closed quadratic form associated to \eqref{deg} we can use Th. \ref{mainweapon} to conclude that $$\lim_{n\rightarrow \infty}\lambda_k(s_n)=\lambda_k(0)$$ Since $\{s_n\}_{n\in \mathbb{N}}\subset (0,1]$ is any arbitrary sequence with $s_n\rightarrow 0$ as $n\rightarrow \infty$ we can conclude that  $$\lim_{s\rightarrow 0}\lambda_k(s)=\lambda_k(0).$$ Finally the remaining part of Th. \ref{spectralth} follows immediately by Th. \ref{mainweapon}.
\end{proof}

\begin{cor}
 In the setting of Th. \ref{spectralth}. For any positive integer $k$ the sequence $\{\eta_k(z_n)\}_{n\in \mathbb{N}}$, viewed as a sequence where $\eta_k(z_n)\in (\mathcal{D}(Q_{z_n}),Q_{z_n,H})$, converges strongly to  $\eta_k(0)\in (\mathcal{D}(Q_{0}),Q_{0,H})$ in the sense of Def. \ref{strong}.
\end{cor}

\begin{proof}
By Prop. \ref{extension} and \ref{pakistan}, in order to verify Def. \ref{strong}, it is enough to prove that 
\begin{equation}
\label{fury}
\lim_{n\rightarrow \infty}Q_{s_n,H}(\eta_k(0)-\eta_k(z_n),\eta_k(0)-\eta_k(z_n))=0.
\end{equation}
 We have
$$Q_{s_n,H}(\eta_k(0)-\eta_k(z_n),\eta_k(0)-\eta_k(z_n))=\|\eta_k(0)-\eta_k(z_n)\|^2_{L^2\Omega^{m,0}(A,g_{z_n}|_A)}+\|\overline{\partial}_{m,0}\eta_k(0)-\overline{\partial}_{m,0}\eta_k(z_n)\|^2_{L^2\Omega^{m,1}(A,g_{z_n}|_A)}.$$ By Th. \ref{spectralth} we know that $\|\eta_k(0)-\eta_k(z_n)\|^2_{L^2\Omega^{m,0}(A,g_{z_n}|_A)}\rightarrow 0$  as $n\rightarrow \infty$. For the other term, using Prop. \ref{extension}, we have
\begin{align}
& \nonumber \|\overline{\partial}_{m,0}\eta_k(0)-\overline{\partial}_{m,0}\eta_k(z_n)\|^2_{L^2\Omega^{m,1}(A,g_{z_n}|_A)}=
\langle\overline{\partial}_{m,0}\eta_k(0),\overline{\partial}_{m,0}\eta_k(0)\rangle_{L^2\Omega^{m,1}(A,g_{z_n}|_A)}+\\
&\nonumber  +\langle\overline{\partial}_{m,0}\eta_k(z_n),\overline{\partial}_{m,0}\eta_k(z_n)\rangle_{L^2\Omega^{m,1}(A,g_{z_n}|_A)} -\langle\overline{\partial}_{m,0}\eta_k(0),\overline{\partial}_{m,0}\eta_k(z_n)\rangle_{L^2\Omega^{m,1}(A,g_{z_n}|_A)}+\\
&\nonumber -\langle\overline{\partial}_{m,0}\eta_k(z_n),\overline{\partial}_{m,0}\eta_k(0)\rangle_{L^2\Omega^{m,1}(A,g_{z_n}|_A)}= \langle\overline{\partial}_{m,0}\eta_k(0),\overline{\partial}_{m,0}\eta_k(0)\rangle_{L^2\Omega^{m,1}(A,g_{z_n}|_A)}+\\
& \label{edde} +\lambda_{k}(z_n) -\lambda_{k}(z_n)\langle\eta_k(0),\eta_k(z_n)\rangle_{L^2\Omega^{m,0}(A,g_{z_n}|_A)}-\lambda_k(z_n)\langle\eta_k(z_n),\eta_k(0)\rangle_{L^2\Omega^{m,0}(A,g_{z_n}|_A)}.
\end{align}
By Prop. \ref{soothe} we know that $$\langle\overline{\partial}_{m,0}\eta_k(0),\overline{\partial}_{m,0}\eta_k(0)\rangle_{L^2\Omega^{m,1}(A,g_{z_n}|_A)}\rightarrow \langle\overline{\partial}_{m,0,\max}\eta_k(0),\overline{\partial}_{m,0,\max}\eta_k(0)\rangle_{L^2\Omega^{m,1}(A,h|_A)}=\lambda_k(0)$$ as $n\rightarrow \infty$.
Moreover by  Th. \ref{spectralth} we know that both $\lambda_k(z_n)\rightarrow \lambda_k(0)$ and $\lambda_k(z_n)\langle\eta_k(z_n),\eta_k(0)\rangle_{L^2\Omega^{m,0}(A,g_{z_n}|_A)}\rightarrow \lambda_k(0)$ as $n\rightarrow \infty$. 
In conclusion  \eqref{edde} tends to zero as $n\rightarrow \infty$. This shows that \eqref{fury} holds true and so the proof is concluded.
\end{proof}

Now we deal with the  convergence of eigenspaces.  Consider the operator \eqref{deg} and let $$E_1(0),E_2(0),...,E_k(0),...$$ be its eigenspaces with $m_1(0),m_2(0),...,m_k(0),...$ as corresponding multiplicities. Above we have listed the eigenspaces of \eqref{deg} in increasing order with respect to the corresponding eigenvalues, that is given any $\omega\in E_i(0)$ with $\|\omega\|^2_{L^2\Omega^{m,0}(A,h|_A)}=1$ and $\eta\in E_j(0)$ with $\|\eta\|^2_{L^2\Omega^{m,0}(A,h|_A)}=1$ we have $i<j$ if and only if $\|\overline{\partial}_{m,0,\max}\omega\|^2_{L^2\Omega^{m,1}(A,h|_A)}<\|\overline{\partial}_{m,0,\max}\eta\|^2_{L^2\Omega^{m,1}(A,h|_A)}$. Let $P_{E_k(0)}:L^2\Omega^{m,0}(A,h|_A)\rightarrow E_k(0)$ be the corresponding orthogonal projection. Analogously, for each $s\in (0,1]$, consider  the operator \eqref{nondeg} and let $\{\eta_1(s),...,\eta_k(s),...\}$ be any orthonormal base of $L^2\Omega^{m,0}(A,h|_A)$ made by eigensections of \eqref{nondeg} with corresponding eigenvalues $0\leq \lambda_1(s)\leq \lambda_2(s)\leq...\leq \lambda_k(s)\leq...$ . Let $H_k(s)$ be the subspace of $L^2\Omega^{m,0}(A,h|_A)$ defined by 
\begin{equation}
\label{butter}
H_{k}(s):=\text{span}\{\eta_j(s):m_1(0)+...+m_{k-1}(0)<j\leq m_1(0)+...+m_k(0)\}.
\end{equation}
Let $P_{H_k(s)}:L^2\Omega^{m,0}(A,h|_A)\rightarrow H_k(s)$ be the corresponding orthogonal projection. 
\begin{cor}
\label{uniform}
In the setting described above we have
\begin{equation}
\label{sunga}
\lim_{s\rightarrow 0}\|P_{E_k(0)}-P_{H_k(s)}\|_{B(L^2\Omega^{m,0}(A,h|_A))}=0
\end{equation}
that is $P_{H_k(s)}$ converges to $P_{E_k(0)}$  as $s\rightarrow 0$ in the uniform (or norm) operator topology.
\end{cor}

\begin{proof}
Let $\{s_n\}_{n\in \mathbb{N}}\subset (0,1]$ be any sequence with $s_n\rightarrow 0$ as $n\rightarrow \infty$. Thanks to Th. \ref{spectralth} we known that there exists a subsequence $\{z_n\}_{n\in \mathbb{N}}\subset \{s_n\}_{n\in \mathbb{N}}$ such that $\eta_j(z_n)\rightarrow \eta_j(0)$ in $L^2\Omega^{m,0}(A,h|_A)$ as $n\rightarrow \infty$ with $\{\eta_1(0),...,\eta_k(0),...\}$ an orthonormal basis of $L^2\Omega^{m,0}(A,h|_A)$ made by eigenforms of \eqref{deg} with corresponding eigenvalues $0\leq\lambda_1(0)\leq \lambda_2(0)\leq...\leq \lambda_k(0)\leq...$ . Note that the set of eigenforms $\{\eta_j(0): m_1(0)+...+m_{k-1}(0)<j\leq m_1(0)+...+m_{k}(0)\}$ is an orthonormal basis of $E_k(0)$.
Consider now any form  $\omega\in L^2\Omega^{m,0}(A,h|_A)$ with $\|\omega\|_{L^2\Omega^{m,0}(A,h|_A)}=1$. Let us define  $e_{k-1}:=m_1(0)+...+m_{k-1}(0)$, $e_k:=m_1(0)+...+m_{k}(0)$, $a_j(0):=\langle\omega,\eta_j(0)\rangle_{L^2\Omega^{m,0}(A,h|_A)}$ and  $a_j(z_n):=\langle\omega,\eta_j(z_n)\rangle_{L^2\Omega^{m,0}(A,h|_A)}$. Then we have 
\begin{align}
\nonumber & \|P_{E_k(0)}\omega-P_{H_k(s)}\omega\|_{L^2\Omega^k(A,h|_A)}=\|\sum_{j=e_{k-1}+1}^{e_k} a_j(0)\eta_j(0)-\sum_{j=e_{k-1}+1}^{e_k} a_j(z_n)\eta_j(z_n)\|_{L^2\Omega^{m,0}(A,h|_A)}\leq\\
&\nonumber  \sum_{j=e_{k-1}+1}^{e_k}\| a_j(0)\eta_j(0)-a_j(z_n)\eta_j(z_n)\|_{L^2\Omega^{m,0}(A,h|_A)}=\\
& \nonumber \sum_{j=e_{k-1}+1}^{e_k}\| a_j(0)\eta_j(0)-a_j(0)\eta_j(z_n)+a_j(0)\eta_j(z_n)-a_j(z_n)\eta_j(z_n)\|_{L^2\Omega^{m,0}(A,h|_A)}\leq\\
& \nonumber \sum_{j=e_{k-1}+1}^{e_k}\| a_j(0)\eta_j(0)-a_j(0)\eta_j(z_n)\|_{L^2\Omega^{m,0}(A,h|_A)}+\sum_{j=e_{k-1}+1}^{e_k}\|a_j(0)\eta_j(z_n)-a_j(z_n)\eta_j(z_n)\|_{L^2\Omega^{m,0}(A,h|_A)}\leq\\
& \nonumber  \sum_{j=e_{k-1}+1}^{e_k}\| \eta_j(0)-\eta_j(z_n)\|_{L^2\Omega^{m,0}(A,h|_A)}+\sum_{j=e_{k-1}+1}^{e_k} |a_j(0)-a_j(z_n)|\leq 2\sum_{j=e_{k-1}+1}^{e_k}\| \eta_j(0)-\eta_j(z_n)\|_{L^2\Omega^{m,0}(A,h|_A)}
\end{align}
as $|a_j(0)|\leq 1$ and $|a_j(0)-a_j(z_n)|$ $=|\langle\omega,\eta_j(0)\rangle_{L^2\Omega^{m,0}(A,h|_A)}-\langle\omega,\eta_j(z_n)\rangle_{L^2\Omega^{m,0}(A,h|_A)}|$ $=|\langle\omega,\eta_j(0)-\eta_j(z_n)\rangle_{L^2\Omega^{m,0}(A,h|_A)}|$ $\leq \| \eta_j(0)-\eta_j(z_n)\|_{L^2\Omega^{m,0}(A,h|_A)}$.
Therefore $$\|P_{E_k(0)}-P_{H_k(z_n)}\|_{B(L^2\Omega^{m,0}(A,h|_A))}\leq 2\sum_{j=e_{k-1}+1}^{e_k}\| \eta_j(0)-\eta_j(z_n)\|_{L^2\Omega^{m,0}(A,h|_A)}$$ and so we have $$0\leq\lim_{n\rightarrow \infty}\|P_{E_k(0)}-P_{H_k(z_n)}\|_{B(L^2\Omega^{m,0}(A,h|_A))}\leq \lim_{n\rightarrow \infty}2\sum_{j=e_{k-1}+1}^{e_k}\| \eta_j(0)-\eta_j(z_n)\|_{L^2\Omega^{m,0}(A,h|_A)}=0.$$
Assume now that \eqref{sunga} does not hold true. Then there exists a constant $\epsilon>0$ and a sequence $\{s_n\}_{n\in \mathbb{N}}\subset (0,1]$, $s_n\rightarrow 0$ as $n\rightarrow \infty$, such that 
\begin{equation}
\label{sungam}
\lim_{n\rightarrow \infty}\|P_{E_k(0)}-P_{H_k(s_n)}\|_{B(L^2\Omega^{m,0}(A,h|_A))}>\epsilon.
\end{equation}
On the other hand, according to what we have shown above, we can find a subsequence $\{z_n\}_{n\in \mathbb{N}}\subset \{s_n\}_{n\in \mathbb{N}}$ such that 
\begin{equation}
\label{sungamm}
\lim_{n\rightarrow \infty}\|P_{E_k(0)}-P_{H_k(z_n)}\|_{B(L^2\Omega^{m,0}(A,h|_A))}=0
\end{equation}
which clearly contradicts \eqref{sungam}. We can therefore conclude that \eqref{sunga} holds true as desired. 
\end{proof}

\vspace{1 cm}

Now we continue by studying the convergence of the heat operators associated to the family $g_s$.
For each $t\in (0,\infty)$ and $s\in (0,1]$ let 
\begin{equation}
\label{heatnondeg}
e^{-t\Delta_{\overline{\partial},m,0,s}}:L^2\Omega^{m,0}(M,g_s)\rightarrow L^2\Omega^{m,0}(M,g_s)
\end{equation}
be the heat operator associated to \eqref{nondeg}. It is a classical result of elliptic theory on compact manifolds that \eqref{heatnondeg} is a trace class operator. Let
\begin{equation}
\label{heatdeg}
e^{-t\Delta_{\overline{\partial},m,0,\mathrm{abs}}}:L^2\Omega^{m,0}(A,h|_A)\rightarrow L^2\Omega^{m,0}(A,h|_A)
\end{equation}
be the heat operator associated to \eqref{deg}. Thanks to \cite{FBei} Cor. 4.2 we know that \eqref{heatdeg} is trace class too. Let us label with $\Tr(e^{-t\Delta_{\overline{\partial},m,0,s}})$ and $\Tr(e^{-t\Delta_{\overline{\partial},m,0,\mathrm{abs}}})$ the trace of \eqref{heatnondeg} and \eqref{heatdeg}, respectively. We recall that $\Tr(e^{-t\Delta_{\overline{\partial},m,0,s}})=\sum_ke^{-t\lambda_k(s)}$ and analogously $\Tr(e^{-t\Delta_{\overline{\partial},m,0,\mathrm{abs}}})=\sum_ke^{-t\lambda_k(0)}$. Moreover both  $\Tr(e^{-t\Delta_{\overline{\partial},m,0,\mathrm{abs}}})$ and $\Tr(e^{-t\Delta_{\overline{\partial},m,0,s}})$, the latter for any fixed $s\in (0,1]$,  are $C^{\infty}$ functions on $(0,\infty)$. Furthermore we recall the well known fact that, given any separable Hilbert space $H$, the space of trace-class operators, here denoted by $B_1(H)$, is a Banach space with norm $\|A\|_{B_1(H)}:=\Tr|A|$. We have now all the ingredients for the following

\begin{teo}
\label{tracenorm}
Let $t_0\in (0,\infty)$ be arbitrarily fixed. Then $$\lim_{s\rightarrow0}\sup_{t\in [t_0,\infty)}\Tr|e^{-t\Delta_{\overline{\partial},m,0,s}}-e^{-t\Delta_{\overline{\partial},m,0,\mathrm{abs}}}|=0.$$
 Equivalently $e^{-t\Delta_{\overline{\partial},m,0,s}}$ converges to $e^{-t\Delta_{\overline{\partial},m,0,\mathrm{abs}}}$  as $s\rightarrow 0$ with respect to the trace-class norm and uniformly on $[t_0,\infty)$.
\end{teo}

In order to prove the above theorem we need the following property.
\begin{prop}
\label{nunu}
For any $s\in (0,1]$ we have $\ker(\Delta_{\overline{\partial},m,0,s})=\ker(\Delta_{\overline{\partial},m,0,\mathrm{abs}})$.
\end{prop}
\begin{proof}
It is clear that $\ker(\Delta_{\overline{\partial},m,0,s})=\mathcal{K}_M(M)$, with $\mathcal{K}_M(M)$ the space of global sections of the canonical sheaf of $M$ or equivalently the space of holomorphic $(m,0)$-forms on $M$. Consider now $\ker(\Delta_{\overline{\partial},m,0,\mathrm{abs}})$. Then $\ker(\Delta_{\overline{\partial},m,0,\mathrm{abs}})=\ker (\overline{\partial}_{m,0,\max})$ where the latter operator is defined in \eqref{maxx}. By elliptic regularity we have $\ker (\overline{\partial}_{m,0,\max})$ $\subset \Omega^{m,0}(A)$ and by \eqref{equality} we know that $L^2\Omega^{m,0}(A,h|_A)=L^2\Omega^{m,0}(A,g_s|_A)$ for any $s\in (0,1]$. Altogether this shows that $\ker (\overline{\partial}_{m,0,\max})=\{\omega\in \mathcal{K}_M(A)\cap L^2\Omega^{m,0}(A,g_s|_A)\}$. It is therefore immediate to check that $\mathcal{K}_M(M)\subseteq \ker (\overline{\partial}_{m,0,\max})$. Conversely, by \cite{FBei} Prop. 3.2, we have $\{\omega\in \mathcal{K}_M(A)\cap L^2\Omega^{m,0}(A,g_s|_A)\}$ $\subseteq \ker(\overline{\partial}_{m,0})$ where the latter space is the kernel of \eqref{boxuomo}, that is the unique $L^2$ closed extension  of $\overline{\partial}_{m,0}:\Omega^{m,0}(M)\rightarrow \Omega^{m,1}(M)$ with respect to $g_s$. In this way we have $\{\omega\in \mathcal{K}_M(A)\cap L^2\Omega^{m,0}(A,g_s|_A)\}\subseteq \mathcal{K}_M(M)$ as $ \ker(\overline{\partial}_{m,0})=\mathcal{K}_M(M)$.  In conclusion $\mathcal{K}_M(M)=\ker (\overline{\partial}_{m,0,\max})$ and so $\ker(\Delta_{\overline{\partial},m,0,s})=\ker(\Delta_{\overline{\partial},m,0,\mathrm{abs}})$ for any $s\in (0,1]$ as required.
\end{proof}

Now we can prove Th. \ref{tracenorm}.

\begin{proof}
We start with a preliminary remark. As recalled above we know that $\Tr(e^{-t\Delta_{\overline{\partial},m,0,s}})=\sum_ke^{-t\lambda_k(s)}<\infty$.  Let $t_0\in (0,\infty)$ be arbitrarily fixed and let $\nu\in \mathbb{R}$ be as in Prop. \ref{patolax}. Then by \eqref{fiacco} and Th. 4.2 in \cite{FBei} we have 
\begin{equation}
\label{minmax}
\nu\lambda_k(s)\geq \lambda_k(1)
\end{equation}
 for every $s\in [0,1]$ and every positive integer $k$. Therefore for each $\epsilon>0$ and $t_0>0$ arbitrarily fixed there exists $\overline{k}\in \mathbb{N}$ such that 
\begin{equation}
\label{carabrone}
\sum_{k=\overline{k}+1}^{\infty}e^{-t\lambda_k(s)}<\epsilon
\end{equation}
for each $s\in [0,1]$ and $t\in [t_0,\infty)$. From now on let $0<\epsilon<1$ and $t_0\in (0,\infty)$ be arbitrarily fixed. In the rest of the proof we will always assume that $t\in [t_0,\infty]$. Let $\{s_n\}\subset (0,1]$ be any sequence such that $s_n\rightarrow 0$ as $n\rightarrow \infty$. Let  $\{\eta_{1}(s_n),\eta_{2}(s_n),...,\eta_{k}(s_n),...\}$ be any orthonormal basis of $L^2\Omega^{m,0}(A,g_{s_n}|_A)$ made by eigenforms of \eqref{nondeg} with corresponding eigenvalues $\{\lambda_{1}(s_n),\lambda_{2}(s_n),...,\lambda_{k}(s_n),...\}$. Thanks to Th. \ref{spectralth} we know that there exists a subsequence $\{z_n\}\subset \{s_n\}$ and $\{\eta_{1}(0),\eta_{2}(0),...,\eta_{k}(0),...\}$, an orthonormal basis of $L^2\Omega^{m,0}(A,h|_A)$ made by eigenforms of \eqref{deg} with corresponding eigenvalues $\{\lambda_{1}(0),\lambda_{2}(0),$ $...,\lambda_{k}(0),...\}$, such that $\eta_{k}(z_n)\rightarrow \eta_k(0)$ as $n\rightarrow \infty$ in  $L^2\Omega^{m,0}(A,h|_A)$ for each positive integer $k$. Let us fix a positive integer $\overline{k}$ such that \eqref{carabrone} holds true.  Let $P_{\overline{k}}:L^2\Omega^{m,0}(A,h|_A)\rightarrow L^2\Omega^{m,0}(A,h|_A)$ be the orthogonal projection on the subspace generated by $\{\eta_1(0),...,\eta_{\overline{k}}(0)\}$ and let $Q_{\overline{k}}:=\id-P_{\overline{k}}$, with $\id:L^2\Omega^{m,0}(A,h|_A)\rightarrow L^2\Omega^{m,0}(A,h|_A)$ the identity.  We have
\begin{align}
\nonumber & \Tr|e^{-t\Delta_{\overline{\partial},m,0,z_{n}}}-e^{-t\Delta_{\overline{\partial},m,0,\mathrm{abs}}}|=\Tr|(e^{-t\Delta_{\overline{\partial},m,0,z_{n}}}-e^{-t\Delta_{\overline{\partial},m,0,\mathrm{abs}}})\circ (P_{\overline{k}}+Q_{\overline{k}})|\leq\\
& \nonumber  \Tr|(e^{-t\Delta_{\overline{\partial},m,0,z_{n}}}-e^{-t\Delta_{\overline{\partial},m,0,\mathrm{abs}}})\circ P_{\overline{k}}|+\Tr|(e^{-t\Delta_{\overline{\partial},m,0,z_{n}}}-e^{-t\Delta_{\overline{\partial},m,0,\mathrm{abs}}})\circ Q_{\overline{k}}|
\end{align}
Clearly 
\begin{equation}
\label{falco}
\Tr|(e^{-t\Delta_{\overline{\partial},m,0,z_{n}}}-e^{-t\Delta_{\overline{\partial},m,0,\mathrm{abs}}})\circ P_{\overline{k}}|=\sum_{k=1}^{\overline{k}}\langle|(e^{-t\Delta_{\overline{\partial},m,0,z_n}}-e^{-t\Delta_{\overline{\partial},m,0,\mathrm{abs}}})\circ P_{\overline{k}}|\eta_k(0),\eta_k(0) \rangle_{L^2\Omega^{m,0}(A,h|_A)}
\end{equation}
since $|(e^{-t\Delta_{\overline{\partial},m,0,z_n}}-e^{-t\Delta_{\overline{\partial},m,0,\mathrm{abs}}})\circ P_{\overline{k}}|\eta_k(0)=0$ whenever $k>\overline{k}$. Concerning $\Tr|(e^{-t\Delta_{\overline{\partial},m,0,z_{n}}}-e^{-t\Delta_{\overline{\partial},m,0,\mathrm{abs}}})\circ Q_{\overline{k}}|$ we have
\begin{align}
& \nonumber \Tr|(e^{-t\Delta_{\overline{\partial},m,0,z_{n}}}-e^{-t\Delta_{\overline{\partial},m,0,\mathrm{abs}}})\circ Q_{\overline{k}}|\leq \Tr(e^{-t\Delta_{\overline{\partial},m,0,z_{n}}}\circ Q_{\overline{k}})+\Tr(e^{-t\Delta_{\overline{\partial},m,0,\mathrm{abs}}}\circ Q_{\overline{k}})=\\
& \nonumber \sum_{k=\overline{k}+1}^{\infty}\langle e^{-t\Delta_{\overline{\partial},m,0,\mathrm{abs}}}\eta_k(0),\eta_k(0) \rangle_{L^2\Omega^{m,0}(A,h|_A)}+\Tr(e^{-t\Delta_{\overline{\partial},m,0,z_n}}\circ Q_{\overline{k}})=\\
& \label{dede} \sum_{k=\overline{k}+1}^{\infty}e^{-t\lambda_k(0)}+\Tr(e^{-t\Delta_{\overline{\partial},m,0,z_n}}\circ Q_{\overline{k}}) \leq \epsilon+\Tr(e^{-t\Delta_{\overline{\partial},m,0,z_n}}\circ Q_{\overline{k}})\ (\text{thanks to}\ \eqref{carabrone}).
\end{align}
Concerning $\Tr(e^{-t\Delta_{\overline{\partial},m,0,z_n}}\circ Q_{\overline{k}})$ we have 
\begin{align}
& \nonumber \Tr(e^{-t\Delta_{\overline{\partial},m,0,z_n}}\circ Q_{\overline{k}})=\sum_{k=\overline{k}+1}^{\infty}\langle e^{-t\Delta_{\overline{\partial},m,0,z_n}} \eta_k(0),\eta_k(0) \rangle_{L^2\Omega^{m,0}(A,h|_A)}=\\
& \nonumber \sum_{k=1}^{\infty}\langle e^{-t\Delta_{\overline{\partial},m,0,z_n}} \eta_k(0),\eta_k(0) \rangle_{L^2\Omega^{m,0}(A,h|_A)}-\sum_{k=1}^{\overline{k}}\langle e^{-t\Delta_{\overline{\partial},m,0,z_n}} \eta_k(0),\eta_k(0) \rangle_{L^2\Omega^{m,0}(A,h|_A)}=\\
& \nonumber \Tr(e^{-t\Delta_{\overline{\partial},m,0,z_n}})-\sum_{k=1}^{\overline{k}}\langle e^{-t\Delta_{\overline{\partial},m,0,z_n}} \eta_k(0),\eta_k(0) \rangle_{L^2\Omega^{m,0}(A,h|_A)}=\\
& \nonumber \sum_{k=1}^{\infty}\langle e^{-t\Delta_{\overline{\partial},m,0,z_n}} \eta_k(z_n),\eta_k(z_n) \rangle_{L^2\Omega^{m,0}(A,h|_A)}-\sum_{k=1}^{\overline{k}}\langle e^{-t\Delta_{\overline{\partial},m,0,z_n}} \eta_k(0),\eta_k(0) \rangle_{L^2\Omega^{m,0}(A,h|_A)}=
\end{align}

\begin{align}
& \nonumber \sum_{k=1}^{\overline{k}}\langle e^{-t\Delta_{\overline{\partial},m,0,z_n}} \eta_k(z_n),\eta_k(z_n) \rangle_{L^2\Omega^{m,0}(A,h|_A)}+\sum_{k=\overline{k}+1}^{\infty}\langle e^{-t\Delta_{\overline{\partial},m,0,z_n}} \eta_k(z_n),\eta_k(z_n) \rangle_{L^2\Omega^{m,0}(A,h|_A)}-\\
& \nonumber \sum_{k=1}^{\overline{k}}\langle e^{-t\Delta_{\overline{\partial},m,0,z_n}} \eta_k(0),\eta_k(0) \rangle_{L^2\Omega^{m,0}(A,h|_A)}=\sum_{k=1}^{\overline{k}}\langle e^{-t\Delta_{\overline{\partial},m,0,z_n}} \eta_k(z_n),\eta_k(z_n) \rangle_{L^2\Omega^{m,0}(A,h|_A)}+\\ 
\nonumber & \sum_{k=\overline{k}+1}^{\infty}e^{-t\lambda_k(z_n)}-
 \sum_{k=1}^{\overline{k}}\langle e^{-t\Delta_{\overline{\partial},m,0,z_n}} \eta_k(0),\eta_k(0) \rangle_{L^2\Omega^{m,0}(A,h|_A)}\leq\ (\text{again by}\ \eqref{carabrone}) \\
& \label{derede} \epsilon+ \sum_{k=1}^{\overline{k}}\langle e^{-t\Delta_{\overline{\partial},m,0,z_n}} \eta_k(z_n),\eta_k(z_n) \rangle_{L^2\Omega^{m,0}(A,h|_A)}-
 \sum_{k=1}^{\overline{k}}\langle e^{-t\Delta_{\overline{\partial},m,0,z_n}} \eta_k(0),\eta_k(0) \rangle_{L^2\Omega^{m,0}(A,h|_A)}.
\end{align}

Above we have used the well known property that the trace of a positive self-adjoint trace-class operator is independent on the orthonormal basis. Finally we have
 
\begin{align}
& \nonumber \sum_{k=1}^{\overline{k}}\langle e^{-t\Delta_{\overline{\partial},m,0,z_n}} \eta_k(z_n),\eta_k(z_n) \rangle_{L^2\Omega^{m,0}(A,h|_A)}-\sum_{k=1}^{\overline{k}}\langle e^{-t\Delta_{\overline{\partial},m,0,z_n}} \eta_k(0),\eta_k(0) \rangle_{L^2\Omega^{m,0}(A,h|_A)}=\\
& \nonumber \sum_{k=1}^{\overline{k}}\langle e^{-t\Delta_{\overline{\partial},m,0,z_n}} \eta_k(z_n),\eta_k(0)-\eta_k(0)+\eta_k(z_n) \rangle_{L^2\Omega^{m,0}(A,h|_A)}- \sum_{k=1}^{\overline{k}}\langle e^{-t\Delta_{\overline{\partial},m,0,z_n}} \eta_k(0),\eta_k(0) \rangle_{L^2\Omega^{m,0}(A,h|_A)}=
\end{align}
\begin{align}
\nonumber & \sum_{k=1}^{\overline{k}}\langle e^{-t\Delta_{\overline{\partial},m,0,z_n}} \eta_k(z_n),\eta_k(z_n)-\eta_k(0) \rangle_{L^2\Omega^{m,0}(A,h|_A)}+ \sum_{k=1}^{\overline{k}}\langle e^{-t\Delta_{\overline{\partial},m,0,z_n}}(\eta_k(z_n)- \eta_k(0)),\eta_k(0) \rangle_{L^2\Omega^{m,0}(A,h|_A)}\leq\\
& \nonumber \sum_{k=1}^{\overline{k}}\|e^{-t\Delta_{\overline{\partial},m,0,z_n}} \eta_k(z_n)\|_{L^2\Omega^{m,0}(A,h|_A)}\|\eta_k(z_n)-\eta_k(0) \|_{L^2\Omega^{m,0}(A,h|_A)}+ \sum_{k=1}^{\overline{k}}\|e^{-t\Delta_{\overline{\partial},m,0,z_n}}(\eta_k(z_n)- \eta_k(0))\|_{L^2\Omega^{m,0}(A,h|_A)}\leq\\
& \label{tobrarex} 2\sum_{k=1}^{\overline{k}}\|e^{-t\Delta_{\overline{\partial},m,0,z_n}}\|_{B(L^2\Omega^{m,0}(A,h|_A))}\|\eta_k(z_n)- \eta_k(0)\|_{L^2\Omega^{m,0}(A,h|_A)}.
\end{align}
Above we have denoted by $\|e^{-t\Delta_{\overline{\partial},m,0,z_n}}\|_{B(L^2\Omega^{m,0}(A,h|_A))}$ the norm of the operator $e^{-t\Delta_{\overline{\partial},m,0,z_n}}:L^2\Omega^{m,0}(A,h|_A)\rightarrow L^2\Omega^{m,0}(A,h|_A)$. 
It is clear that   $\|e^{-t\Delta_{\overline{\partial},m,0,s}}\|_{B(L^2\Omega^{m,0}(A,h|_A))}\leq 1$ for any $s\in [0,1]$.
This can be deduced immediately by the fact that, since $e^{-t\Delta_{\overline{\partial},m,0,s}}:L^2\Omega^{m,0}(A,h|_A))\rightarrow L^2\Omega^{m,0}(A,h|_A)$ is positive and self-adjoint, we have $$\|e^{-t\Delta_{\overline{\partial},m,0,s}}\|_{B(L^2\Omega^{m,0}(A,h|_A))}=\sup_{\omega\neq 0}\frac{\langle e^{-t\Delta_{\overline{\partial},m,0,s}}\omega,\omega\rangle_{L^2\Omega^{m,0}(A,h|_A)}}{\|\omega\|^2_{L^2\Omega^{m,0}(A,h|_A)}}.$$
Summarizing, by \eqref{falco}-\eqref{dede}-\eqref{derede}-\eqref{tobrarex}, we showed for the moment that
\begin{align}
& \nonumber \Tr|e^{-t\Delta_{\overline{\partial},m,0,z_{n}}}-e^{-t\Delta_{\overline{\partial},m,0,\mathrm{abs}}}|\leq\sum_{k=1}^{\overline{k}}\langle|(e^{-t\Delta_{\overline{\partial},m,0,z_n}}-e^{-t\Delta_{\overline{\partial},m,0,\mathrm{abs}}})\circ P_{\overline{k}}|\eta_k(0),\eta_k(0) \rangle_{L^2\Omega^{m,0}(A,h|_A)}+\\
& \label{rambaldo} 2\epsilon+2\sum_{k=1}^{\overline{k}}\|\eta_k(z_n)- \eta_k(0)\|_{L^2\Omega^{m,0}(A,h|_A)}.
\end{align}
Concerning $\sum_{k=1}^{\overline{k}}\langle|(e^{-t\Delta_{\overline{\partial},m,0,z_n}}-e^{-t\Delta_{\overline{\partial},m,0,\mathrm{abs}}})\circ P_{\overline{k}}|\eta_k(0),\eta_k(0) \rangle_{L^2\Omega^{m,0}(A,h|_A)}$ we have
\begin{align}
\nonumber & \sum_{k=1}^{\overline{k}}\langle|(e^{-t\Delta_{\overline{\partial},m,0,z_n}}-e^{-t\Delta_{\overline{\partial},m,0,\mathrm{abs}}})\circ P_{\overline{k}}|\eta_k(0),\eta_k(0) \rangle_{L^2\Omega^{m,0}(A,h|_A)}\leq\\
&\nonumber \sum_{k=1}^{\overline{k}}\| |(e^{-t\Delta_{\overline{\partial},m,0,z_n}}-e^{-t\Delta_{\overline{\partial},m,0,\mathrm{abs}}})\circ P_{\overline{k}}|\eta_k(0)\|_{L^2\Omega^{m,0}(A,h|_A)}=\\
& \nonumber \sum_{k=1}^{\overline{k}}\|(e^{-t\Delta_{\overline{\partial},m,0,z_n}}-e^{-t\Delta_{\overline{\partial},m,0,\mathrm{abs}}})(P_{\overline{k}}\eta_k(0))\|_{L^2\Omega^{m,0}(A,h|_A)}=\\
& \nonumber \sum_{k=1}^{\overline{k}}\|(e^{-t\Delta_{\overline{\partial},m,0,z_n}}-e^{-t\Delta_{\overline{\partial},m,0,\mathrm{abs}}})(\eta_k(0))\|_{L^2\Omega^{m,0}(A,h|_A)}= \sum_{k=1}^{\overline{k}}\|e^{-t\Delta_{\overline{\partial},m,0,z_n}}\eta_k(0)-e^{-t\lambda_k(0)}\eta_k(0)\|_{L^2\Omega^{m,0}(A,h|_A)}=\\
& \nonumber \sum_{k=1}^{\overline{k}}\|e^{-t\Delta_{\overline{\partial},m,0,z_n}}(\eta_k(z_n)+\eta_k(0)-\eta_k(z_n))-e^{-t\lambda_k(0)}\eta_k(0)\|_{L^2\Omega^{m,0}(A,h|_A)}\leq\\
& \nonumber \sum_{k=1}^{\overline{k}}\|e^{-t\lambda_k(z_n)}\eta_k(z_n)-e^{-t\lambda_k(0)}\eta_k(0)\|_{L^2\Omega^{m,0}(A,h|_A)}+\sum_{k=1}^{\overline{k}}\|e^{-t\Delta_{\overline{\partial},m,0,z_n}}(\eta_k(0)-\eta_k(z_n))\|_{L^2\Omega^{m,0}(A,h|_A)}\leq\\
& \nonumber \sum_{k=1}^{\overline{k}}\|e^{-t\lambda_k(z_n)}\eta_k(z_n)-e^{-t\lambda_k(z_n)}\eta_k(0)+e^{-t\lambda_k(z_n)}\eta_k(0)-e^{-t\lambda_k(0)}\eta_k(0)\|_{L^2\Omega^{m,0}(A,h|_A)}+\\
\nonumber &\sum_{k=1}^{\overline{k}}\|\eta_k(0)-\eta_k(z_n)\|_{L^2\Omega^{m,0}(A,h|_A)}\leq\\
& \nonumber \sum_{k=1}^{\overline{k}}e^{-t\lambda_k(z_n)}\|\eta_k(z_n)-\eta_k(0)\|_{L^2\Omega^{m,0}(A,h|_A)}+\sum_{k=1}^{\overline{k}}|e^{-t\lambda_k(z_n)}-e^{-t\lambda_k(0)}|+\sum_{k=1}^{\overline{k}}\|\eta_k(0)-\eta_k(z_n)\|_{L^2\Omega^{m,0}(A,h|_A)}\leq\\
& \label{crawl}2\sum_{k=1}^{\overline{k}}\|\eta_k(z_n)-\eta_k(0)\|_{L^2\Omega^{m,0}(A,h|_A)}+\sum_{k=\ell+1}^{\overline{k}}|e^{-t\lambda_k(z_n)}-e^{-t\lambda_k(0)}|.
\end{align}
In \eqref{crawl} we have denoted $\ell:=\dim(\ker(\Delta_{\overline{\partial},m,0,\mathrm{abs}}))$ and we have used Prop. \ref{nunu} for the equality $\sum_{k=\ell+1}^{\overline{k}}|e^{-t\lambda_k(z_n)}-e^{-t\lambda_k(0)}|=\sum_{k=1}^{\overline{k}}|e^{-t\lambda_k(z_n)}-e^{-t\lambda_k(0)}|$. Note that $\ell<\overline{k}$ as $0<\epsilon<1$. Joining \eqref{rambaldo} and \eqref{crawl} we have finally achieved the upper estimate we were looking for:
\begin{equation}
\label{babel}
 \Tr|e^{-t\Delta_{\overline{\partial},m,0,z_{n}}}-e^{-t\Delta_{\overline{\partial},m,0,\mathrm{abs}}}|\leq  \sum_{k=\ell+1}^{\overline{k}}|e^{-t\lambda_k(z_n)}-e^{-t\lambda_k(0)}|+ 2\epsilon+4\sum_{k=1}^{\overline{k}}\|\eta_k(z_n)- \eta_k(0)\|_{L^2\Omega^{m,0}(A,h|_A)}.
\end{equation}

By Th. \ref{spectralth}, \eqref{minmax} and the fact that the function $e^{-x}$ is $1$-Lipschitz on $[0,\infty)$ we can find $t_1$, with $t_0\leq t_1<\infty$ and $n_0>0$ such that: 
\begin{enumerate}
\item $\|\eta_k(z_n)-\eta_k(0)\|_{L^2\Omega^{m,0}(A,h|_A)}\leq \epsilon/\overline{k}$ for any $n>n_0$ and $k=1,...,\overline{k}$, 
\item $|e^{-t\lambda_k(z_n)}-e^{-t\lambda_k(0)}|\leq e^{-t\lambda_k(z_n)}+e^{-t\lambda_k(0)} \leq \epsilon/\overline{k}$ for any $t_1\leq t<\infty$, $n\in \mathbb{N}$ and $k=\ell+1,...,\overline{k}$, 
\item $|e^{-t\lambda_k(z_n)}-e^{-t\lambda_k(0)}|\leq t|\lambda_k(z_n)-\lambda_k(0)|\leq t_1|\lambda_k(z_n)-\lambda_k(0)|\leq \epsilon/\overline{k}$ for any $k=\ell+1,...,\overline{k}$, $n>n_0$ and $t_0\leq t\leq t_1$.
\end{enumerate}
Note that the second point above follows quickly by \eqref{minmax} since $e^{-t\lambda_k(z_n)}+e^{-t\lambda_k(0)} \leq 2e^{-\frac{t}{\nu}\lambda_k(1)}$. Moreover we point out that the second and the third point above yield  $$|e^{-t\lambda_k(z_n)}-e^{-t\lambda_k(0)}|\leq \epsilon/\overline{k}$$ for each  $k=\ell+1,...,\overline{k}$, $t\in [t_0,\infty)$ and $n>n_0$.
Thus, thanks to \eqref{babel} and the above three remarks, we have shown that $$\Tr|e^{-t\Delta_{\overline{\partial},m,0,z_{n}}}-e^{-t\Delta_{\overline{\partial},m,0,\mathrm{abs}}}|\leq 7\epsilon$$ for each $n>n_0$ and $t\in [t_0,\infty)$.
Summarizing we have proved that given any sequence $\{s_n\}\subset (0,1]$ with $s_n\rightarrow 0$ as $n\rightarrow \infty$ there is a subsequence $\{z_n\}\subset \{s_n\}$ such that for any arbitrarily fixed $0<\epsilon<1$ and $t_0>0$ there exists a positive integer $n_0$ such that $$\sup_{t\in[t_0,\infty)}\Tr|e^{-t\Delta_{\overline{\partial},m,0,z_{n}}}-e^{-t\Delta_{\overline{\partial},m,0,\mathrm{abs}}}|\leq 7\epsilon$$ for any $n>n_0$. 
Clearly this amounts to saying that given any sequence $\{s_n\}\subset (0,1]$ with $s_n\rightarrow 0$ as $n\rightarrow \infty$ there is a subsequence $\{z_n\}\subset \{s_n\}$ such that, for any arbitrarily fixed $t_0>0$, we have
\begin{equation}
\label{pocoloco}
\lim_{n\rightarrow \infty}\sup_{t\in [t_0,\infty)}\Tr|e^{-t\Delta_{\overline{\partial},m,0,z_n}}-e^{-t\Delta_{\overline{\partial},m,0,\mathrm{abs}}}|=0.
\end{equation}
Finally  it is immediate to check that \eqref{pocoloco} is in turn equivalent to saying that for any arbitrarily fixed $t_0>0$ we have $$\lim_{s\rightarrow 0}\sup_{t\in [t_0,\infty)}\Tr|e^{-t\Delta_{\overline{\partial},m,0,s}}-e^{-t\Delta_{\overline{\partial},m,0,\mathrm{abs}}}|=0.$$ The theorem is thus proved.
\end{proof}

\begin{cor}
\label{traceconver}
Let $t_0\in (0,\infty)$ be arbitrarily fixed. Then $$\lim_{s\rightarrow0}\Tr(e^{-t\Delta_{\overline{\partial},m,0,s}})=\Tr(e^{-t\Delta_{\overline{\partial},m,0,\mathrm{abs}}})$$
uniformly on $[t_0,\infty)$.
\end{cor}
\begin{proof}
This is an immediate consequence of Th. \ref{tracenorm}. 
\end{proof}

Now we study the behavior as $s\rightarrow 0$ of the corresponding zeta functions.  As in the previous proof let $\ell:=\dim(\ker(\Delta_{\overline{\partial},m,0,\mathrm{abs}}))$. First of all we want to show that   $\Tr(e^{-t\Delta_{\overline{\partial},m,0,\mathrm{abs}}})-\ell$ decays exponentially as $t\rightarrow \infty$. The argument is essentially the same that is used in the smooth compact  case. We recall it for the sake of completeness. Let $t\in [1,\infty)$. We have 
\begin{align}
&\nonumber \Tr(e^{-t\Delta_{\overline{\partial},m,0,\mathrm{abs}}})-\ell=\sum_{k=\ell+1}^{\infty}e^{-t\lambda_k(0)}=\sum_{k=\ell+1}^{\infty}e^{-\frac{t}{2}\lambda_k(0)}e^{-\frac{t}{2}\lambda_k(0)}\leq \sum_{k=\ell+1}^{\infty}e^{-\frac{t}{2}\lambda_k(0)}e^{-\frac{t}{2}\lambda_{\ell+1}(0)}\leq\\
&\nonumber e^{-\frac{t}{2}\lambda_{\ell+1}(0)}\sum_{k=\ell+1}^{\infty}e^{-\frac{1}{2}\lambda_k(0)}.
\end{align}
In conclusion for any  $t\in [1,\infty)$ we have 
\begin{equation}
\label{expdecay}
\Tr(e^{-t\Delta_{\overline{\partial},m,0,\mathrm{abs}}})-\ell\leq Ae^{-\frac{t}{2}\lambda_{\ell+1}(0)}
\end{equation}
 with $A=\sum_{k=\ell+1}^{\infty}e^{-\frac{1}{2}\lambda_k(0)}$. As $\lambda_{\ell+1}(0)>0$ the above inequality shows that $\Tr(e^{-t\Delta_{\overline{\partial},m,0,\mathrm{abs}}})-\ell$ decays exponentially as $t\rightarrow \infty$.
Moreover  thanks to \cite{FBei} Cor. 4.2 we know that $\Tr(e^{-t\Delta_{\overline{\partial},m,0,\mathrm{abs}}})\leq Gt^{-m}$ for each $t\in (0,1]$ and a positive constant $G$. Therefore for any $x\in \mathbb{C}$ with $\mathrm{Re}(x)>m$ the following integral
$$\frac{1}{\Gamma(x)}\int_0^{\infty}t^{x-1}(\Tr(e^{-t\Delta_{\overline{\partial},m,0,\mathrm{abs}}})-\ell)dt$$ converges and defines a holomorphic function $\zeta(\Delta_{\overline{\partial},m,0,\mathrm{abs}})(x)$ on $\{x\in \mathbb{C}: \mathrm{Re}(x)>m\}$. Furthermore by \cite{FBei} Th. 4.2 and the fact that $\Gamma(x)\lambda_k^{-x}(0)=\int_0^{\infty}t^{x-1}e^{-t\lambda_k(0)}dt$ for any $k>\ell$ we have that $\sum_{k=\ell+1}^{\infty}\lambda_k^{-x}(0)$ converges on $\{x\in \mathbb{C}: \mathrm{Re}(x)>m\}$ and $$\zeta(\Delta_{\overline{\partial},m,0,\mathrm{abs}})(x)=\frac{1}{\Gamma(x)}\int_0^{\infty}t^{x-1}(\Tr(e^{-t\Delta_{\overline{\partial},m,0,\mathrm{abs}}})-\ell)dt=\sum_{k=\ell+1}^{\infty}\lambda_k^{-x}(0)$$ on $\{x\in \mathbb{C}: \mathrm{Re}(x)>m\}$. Finally for any $s\in (0,1]$ let us label by $\zeta(\Delta_{\overline{\partial},m,0,s})(x)$
the zeta function of $\Delta_{\overline{\partial},m,0,s}:L^2\Omega^{m,0}(M,g_s)\rightarrow L^2\Omega^{m,0}(M,g_s)$. We have the following property:

\begin{teo}
\label{Zeta}
Let  $\zeta(\Delta_{\overline{\partial},m,0,\mathrm{abs}})(x)$ and $\zeta(\Delta_{\overline{\partial},m,0,s})(x)$ be the zeta functions of $\Delta_{\overline{\partial},m,0,\mathrm{abs}}$ and $\Delta_{\overline{\partial},m,0,s}$, respectively. Then $$\lim_{s\rightarrow 0}\zeta(\Delta_{\overline{\partial},m,0,s})(x)=\zeta(\Delta_{\overline{\partial},m,0,\mathrm{abs}})(x)$$ for any  $x\in \mathbb{C}$ with $\mathrm{Re}(x)>m$. Moreover the convergence is uniform on any compact subset $K$ of $\{x\in \mathbb{C}:$ $\mathrm{Re}(x)>m\}$.
\end{teo}

\begin{proof}
First of all we need to develop some uniform estimates that will be used along the proof. Thanks to \eqref{minmax} we know that $\nu\lambda_k(s)\geq \lambda_k(1)$ for any $s\in [0,1]$. Therefore $$\sum_{k=1}^{\infty} e^{-t\nu\lambda_k(s)}\leq \sum_{k=1}^{\infty} e^{-t\lambda_k(1)} $$ for any $s\in [0,1]$ and $t\in (0,\infty)$. Thus, using the well known fact that $\Tr(e^{-t\Delta_{\overline{\partial},m,0,1}})\leq Ct^{-m}$ for $t\in (0,1]$ and some constant $C>0$, see for instance \cite{BGV} Th. 2.41,  we obtain that 
\begin{equation}
\sum_{k=1}^{\infty} e^{-t\lambda_k(s)}\leq \nu^{m}Ct^{-m}
\label{uniformupper}
\end{equation}
for any $s\in [0,1]$ and $t\in (0,\nu]$.
Moreover for $t\in [1,\infty)$, $s\in [0,1]$ and a fixed positive constant $B$ we have $\Tr(e^{-t\nu\Delta_{\overline{\partial},m,0,s}})-\ell\leq \Tr(e^{-t\Delta_{\overline{\partial},m,0,1}})-\ell\leq B e^{-\frac{t}{2}\lambda_{\ell+1}(1)}$ from which we deduce 
\begin{equation}
\label{largetime}
 \Tr(e^{-t\Delta_{\overline{\partial},m,0,s}})-\ell\leq  B e^{-\frac{t}{2\nu}\lambda_{\ell+1}(1)}
\end{equation}
for $t\in [\nu,\infty)$ and $s\in [0,1]$.  Now let $K$ be a compact subset of $\{x\in \mathbb{C}:\mathrm{Re}(x)>m\}$. Let $\epsilon>0$ be arbitrarily fixed. Given any $0<\mu\leq 1$ and $x\in K$ we have
\begin{align}
\nonumber & |\zeta(\Delta_{\overline{\partial},m,0,\mathrm{abs}})(x)-\zeta(\Delta_{\overline{\partial},m,0,s})(x)|\leq \frac{1}{|\Gamma(x)|}\int_0^{\infty}|t^{x-1}||\Tr(e^{-t\Delta_{\overline{\partial},m,0,\mathrm{abs}}})-\Tr(e^{-t\Delta_{\overline{\partial},m,0,s}})|dt=\\
& \nonumber \frac{1}{|\Gamma(x)|}\int_0^{\mu}|t^{x-1}||\Tr(e^{-t\Delta_{\overline{\partial},m,0,\mathrm{abs}}})-\Tr(e^{-t\Delta_{\overline{\partial},m,0,s}})|dt+\frac{1}{|\Gamma(x)|}\int_{\mu}^{\nu}|t^{x-1}||\Tr(e^{-t\Delta_{\overline{\partial},m,0,\mathrm{abs}}})-\Tr(e^{-t\Delta_{\overline{\partial},m,0,s}})|dt\\
& \nonumber +\frac{1}{|\Gamma(x)|}\int_\nu^{\infty}|t^{x-1}||\Tr(e^{-t\Delta_{\overline{\partial},m,0,\mathrm{abs}}})-\Tr(e^{-t\Delta_{\overline{\partial},m,0,s}})|dt.
\end{align}
Let us examine in details the above three integrals. By \eqref{uniformupper} we know that  
$|\Tr(e^{-t\Delta_{\overline{\partial},m,0,\mathrm{abs}}})-\Tr(e^{-t\Delta_{\overline{\partial},m,0,s}})|\leq 2\beta t^{-m}$ for every $t\in (0,\nu]$, $s\in [0,1]$ and  with $\beta:=C\nu^{m}$. Let us define $\alpha:= \max \{|\Gamma^{-1}(x)|:x\in K\}$ and $a=\min \{\mathrm{Re}(x):x\in K\}$. For any $x\in K$ with $x=x_1+ix_2$ and $0<t\leq1$ we have $|t^x|=|t^{x_1+ix_2}|\leq 2t^a$. Therefore, taking $\mu=\min\left\{1,\left(\frac{(a-m)\epsilon}{4\alpha\beta}\right)^{\frac{1}{a-m}}\right\}$ and   keeping in mind that $a>m$, we have
\begin{align}
& \nonumber \frac{1}{|\Gamma(x)|}\int_0^{\mu}|t^{x-1}||\Tr(e^{-t\Delta_{\overline{\partial},m,0,\mathrm{abs}}})-\Tr(e^{-t\Delta_{\overline{\partial},m,0,s}})|dt\leq 2\alpha\int_0^{\mu}|t^{x_1+ix_2-1}|\beta t^{-m}dt\leq 4\alpha\int_0^{\mu}t^{a-1}\beta t^{-m}dt=\\
&\nonumber 4\alpha\beta\int_0^{\mu}t^{a-m-1}dt=4\alpha\beta\frac{\mu^{a-m}}{a-m}\leq \epsilon.
\end{align}
Let us now define $b:=\max \{|t^{x-1}|: (t,x)\in [\mu,\nu]\times K\}$. In this way we can deduce the following estimate
$$\frac{1}{|\Gamma(x)|}\int_{\mu}^{\nu}|t^{x-1}||\Tr(e^{-t\Delta_{\overline{\partial},m,0,\mathrm{abs}}})-\Tr(e^{-t\Delta_{\overline{\partial},m,0,s}})|dt\leq \alpha \int_{\mu}^{\nu}b|\Tr(e^{-t\Delta_{\overline{\partial},m,0,\mathrm{abs}}})-\Tr(e^{-t\Delta_{\overline{\partial},m,0,s}})|dt$$
and by Cor. \ref{traceconver} we can find a sufficiently small positive $\delta_1$ such that $$\alpha \int_{\mu}^{\nu}b|\Tr(e^{-t\Delta_{\overline{\partial},m,0,\mathrm{abs}}})-\Tr(e^{-t\Delta_{\overline{\partial},m,0,s}})|dt\leq \epsilon$$ for any $s\in [0,\delta_1]$. Finally let $\xi:=\max \{\mathrm{Re}(x):x\in K\}$. Then $|t^{x-1}|\leq 2t^{\xi-1}$ for any $x\in K$ and so, thanks to \eqref{largetime}, there exist positive constants $\theta$ and $\sigma$ such that $$|t^{x-1}||\Tr(e^{-t\Delta_{\overline{\partial},m,0,\mathrm{abs}}})-\Tr(e^{-t\Delta_{\overline{\partial},m,0,s}})|\leq 2t^{\xi-1}|\Tr(e^{-t\Delta_{\overline{\partial},m,0,\mathrm{abs}}})-\Tr(e^{-t\Delta_{\overline{\partial},m,0,s}})|\leq 4\theta t^{\xi-1}e^{-\sigma t}$$ for any $t\in [\nu,\infty)$, $x\in K$ and $s\in[0,1]$. In particular we get
\begin{align}
& \nonumber \int_{\nu}^{\infty}|t^{x-1}||\Tr(e^{-t\Delta_{\overline{\partial},m,0,\mathrm{abs}}})-\Tr(e^{-t\Delta_{\overline{\partial},m,0,s}})|dt\leq \int_{\nu}^{\infty}2t^{\xi-1}|\Tr(e^{-t\Delta_{\overline{\partial},m,0,\mathrm{abs}}})-\Tr(e^{-t\Delta_{\overline{\partial},m,0,s}})|dt\leq\\
& \nonumber \int_{\nu}^{\infty}4\theta t^{\xi-1}e^{-\sigma t}dt<\infty
\end{align}
for any $x\in K$ and $s\in [0,1]$. Hence, using  the Lebesgue dominate convergence theorem and Cor. \ref{traceconver},  we obtain
\begin{align}
& \nonumber 0\leq \lim_{s\rightarrow 0} \frac{1}{|\Gamma(x)|}\int_\nu^{\infty}|t^{x-1}||\Tr(e^{-t\Delta_{\overline{\partial},m,0,\mathrm{abs}}})-\Tr(e^{-t\Delta_{\overline{\partial},m,0,s}})|dt \leq\\
&\nonumber \lim_{s\rightarrow 0}2\alpha\int_\nu^{\infty}t^{\xi-1}|\Tr(e^{-t\Delta_{\overline{\partial},m,0,\mathrm{abs}}})-\Tr(e^{-t\Delta_{\overline{\partial},m,0,s}})|dt=\\
&\nonumber 2\alpha\int_\nu^{\infty}\lim_{s\rightarrow 0}t^{\xi-1}|\Tr(e^{-t\Delta_{\overline{\partial},m,0,\mathrm{abs}}})-\Tr(e^{-t\Delta_{\overline{\partial},m,0,s}})|dt=0.
\end{align}
By the above limit we deduce that there is a positive number $\delta_2>0$ such that for any $x\in K$ and $s\in [0,\delta_2]$  the following inequality holds true $$\frac{1}{|\Gamma(x)|}\int_\nu^{\infty}|t^{x-1}||\Tr(e^{-t\Delta_{\overline{\partial},m,0,\mathrm{abs}}})-\Tr(e^{-t\Delta_{\overline{\partial},m,0,s}})|dt\leq \epsilon.$$
Summing up we proved that for any arbitrarily fixed $\epsilon >0$ there exists a sufficiently small positive $\delta>0$ such that for any $s\in [0,\delta]$ and $x\in K$ we have $$|\zeta(\Delta_{\overline{\partial},m,0,\mathrm{abs}})(x)-\zeta(\Delta_{\overline{\partial},m,0,s})(x)|\leq 3\epsilon.$$ The statement of this theorem is now an immediate consequence of the above inequality.
 \end{proof}

Finally we come to the last result of this paper. First we recall very briefly  some well known result about the heat kernel. This latter topic is thoroughly studied in many books and papers.  We refer for instance to \cite{BGV}, \cite{Gilkey}, \cite{GYA}, \cite{Gun} and the reference therein. In particular the statements below follow by arguing as in \cite{Gilkey}. Let us label by $K_A=K_M|_A$ the canonical bundle of $A$. Consider the left and the right projections $p_l:A\times A\rightarrow A$ and $p_r:A\times A\rightarrow A$ and let $K_A\boxtimes K_A^*\rightarrow A\times A$ be the vector bundle on $A\times A$ defined by $p_l^*K_A\otimes p_r^*K_A^*$. For any $(x,y)\in A\times A$ the fiber of $K_A\boxtimes K_A^*$ in $(x,y)$ is given by $K_{A_{x}}\otimes K^*_{A_{y}}$ $\cong \mathrm{Hom}(K_{A_{y}},K_{A_{x}})$. We endow the vector bundle $K_A\boxtimes K_A^*\rightarrow A\times A$ with the natural Hermitian metric induced by $h$ and we label it by $\hat{h}$. Moreover on $A\times A$ we consider the product metric induced by $h$. Let $\{\eta_{1,0},\eta_{2,0},...,\eta_{k,0},...\}$ be an orthonormal basis of $L^2\Omega^{m,0}(A,h|_A)$ made by eigenforms of \eqref{deg} with corresponding eigenvalues $\{\lambda_{1,0},\lambda_{2,0},...,\lambda_{k,0},...\}$. For any integer $k$ let $\eta_{k,0}^*\in C^{\infty}(A,K^*_A)$ be the section of $K^*_A$ induced by $\eta_{k,0}$ through $h^*_{m,0}$, that is $ \eta_{k,0}^*:=h_{m,0}^*(\eta_{k,0},\bullet)$. Then it is easy to check that for each $t>0$ the following series 
\begin{equation}
\label{marameo}
\sum_{k=1}^{\infty}e^{-t\lambda_{k,0}}\eta_{k,0}(x)\otimes \eta_{k,0}^*(y)
\end{equation}
converges in $L^2(A\times A,K_A\boxtimes K_A^*)$ and thus it defines an element $K_0(t,x,y)\in L^2(A\times A,K_A\boxtimes K_A^*)$. Moreover by local elliptic estimates and the Sobolev inequality we obtain that for any relatively compact open subset $B$ of $A$ and  positive integers $\ell$ and $j$ the series: $$\sum_{k=1}^{\infty}\lambda_{k,0}^{\ell}e^{-t\lambda_{k,0}}\eta_{k,0}(x)\otimes \eta_{k,0}^*(y)$$ converges over $B$ with respect to the $C^j$-norm and uniformly on $[t_0,\infty)$ with $t_0>0$ arbitrarily fixed. This implies that  $K_0(t,x,y)$ is $C^{\infty}$ with respect to $t$, $x$ and $y$ and moreover that
 $$e^{-t\Delta_{\overline{\partial},m,0,\mathrm{abs}}}\omega=\int_MK_0(t,x,y)\omega(y)\dvol_{h}(y).$$   We are in position to state the next result.

\begin{teo}
\label{HilSCon}
For each $s\in (0,1]$ let $K_s(t,x,y)\in C^{\infty}(M\times M,K_M\boxtimes K_M^*)$ be the heat kernel of $\Delta_{\overline{\partial},m,0,s}:L^2\Omega^{m,0}(M,g_s)\rightarrow L^2\Omega^{m,0}(M,g_s)$. Then for any arbitrarily fixed $t_0>0$ we have $$\lim_{s\rightarrow 0}K_s(t,x,y)=K_0(t,x,y)$$ in $L^2(A\times A, K_A\boxtimes K_A^*)$ and uniformly on $[t_0,\infty)$.
\end{teo}

\begin{proof}
Let $t_0\in (0,\infty)$ and $0<\epsilon<1$ be arbitrarily fixed and let  $\{s_n\}\subset (0,1]$ be any sequence such that $s_n\rightarrow 0$ as $n\rightarrow \infty$. Let  $\{\eta_{1,s_n},\eta_{2,s_n},...,\eta_{k,s_n},...\}$ be an orthonormal basis of $L^2\Omega^{m,0}(A,g_{s_n}|_A)$ made by eigenforms of \eqref{nondeg} with corresponding eigenvalues $\{\lambda_{1,s_n},\lambda_{2,s_n},...,\lambda_{k,s_n},...\}$. Thanks to Th. \ref{spectralth} we know that there exists a subsequence $\{z_n\}\subset \{s_n\}$ and $\{\eta_{1,0},\eta_{2,0},...,\eta_{k,0},...\}$, an orthonormal basis of $L^2\Omega^{m,0}(A,h|_A)$ made by eigenforms of \eqref{deg} with corresponding eigenvalues $\{\lambda_{1,0},\lambda_{2,0},...,\lambda_{k,0},...\}$, such that $\eta_{k,z_n}\rightarrow \eta_{k,0}$ as $n\rightarrow \infty$ in 
$L^2\Omega^{m,0}(A,h|_A)$ for each positive integer $k$. As in \eqref{carabrone} let $\overline{k}$ be a fixed positive integer such that 
\begin{equation}
\label{verdon}
\sum_{k=\overline{k}+1}^{\infty}e^{-t\lambda_{k,s}}\leq \epsilon/2
\end{equation}
for each $s\in [0,1]$.
We have

\begin{align}
\nonumber & \|K_{z_n}(t,x,y)-K_0(t,x,y)\|_{L^2(A\times A,K_A\boxtimes K^*_A)}\leq\\
&\nonumber \|\sum_{k=1}^{\overline{k}}e^{-t\lambda_{k,0}}\eta_{k,0}(x)\otimes \eta_{k,0}^*(y)-\sum_{k=1}^{\overline{k}}e^{-t\lambda_{k,z_n}}\eta_{k,z_n}(x)\otimes \eta_{k,z_n}^*(y)\|_{L^2(A\times A,K_A\boxtimes K^*_A)}+\\
&\nonumber \|\sum_{k=\overline{k}+1}^{\infty}e^{-t\lambda_{k,0}}\eta_{k,0}(x)\otimes \eta_{k,0}^*(y)+\sum_{k=\overline{k}+1}^{\infty}e^{-t\lambda_{k,z_n}}\eta_{k,z_n}(x)\otimes \eta_{k,z_n}^*(y)\|_{L^2(A\times A,K_A\boxtimes K^*_A)}\leq\\
&\nonumber \|\sum_{k=1}^{\overline{k}}e^{-t\lambda_{k,0}}\eta_{k,0}(x)\otimes \eta_{k,0}^*(y)-\sum_{k=1}^{\overline{k}}e^{-t\lambda_{k,z_n}}\eta_{k,z_n}(x)\otimes \eta_{k,z_n}^*(y)\|_{L^2(A\times A,K_A\boxtimes K^*_A)}+\\
&\nonumber \sum_{k=\overline{k}+1}^{\infty}\|e^{-t\lambda_{k,0}}\eta_{k,0}(x)\otimes \eta_{k,0}^*(y)\|_{L^2(A\times A,K_A\boxtimes K^*_A)}+\sum_{k=\overline{k}+1}^{\infty}\|e^{-t\lambda_{k,z_n}}\eta_{k,z_n}(x)\otimes \eta_{k,z_n}^*(y)\|_{L^2(A\times A,K_A\boxtimes K^*_A)}\leq\\
& \nonumber \|\sum_{k=1}^{\overline{k}}e^{-t\lambda_{k,0}}\eta_{k,0}(x)\otimes \eta_{k,0}^*(y)-\sum_{k=1}^{\overline{k}}e^{-t\lambda_{k,z_n}}\eta_{k,z_n}(x)\otimes \eta_{k,z_n}^*(y)\|_{L^2(A\times A,K_A\boxtimes K^*_A)}+\epsilon \leq\\
& \nonumber \sum_{k=1}^{\overline{k}}\|e^{-t\lambda_{k,0}}\eta_{k,0}(x)\otimes \eta_{k,0}^*(y)-e^{-t\lambda_{k,z_n}}\eta_{k,z_n}(x)\otimes \eta_{k,z_n}^*(y)\|_{L^2(A\times A,K_A\boxtimes K^*_A)}+\epsilon.
\end{align}
We can estimates  $\|e^{-t\lambda_{k,0}}\eta_{k,0}(x)\otimes \eta_{k,0}^*(y)-e^{-t\lambda_{k,z_n}}\eta_{k,z_n}(x)\otimes \eta_{k,z_n}^*(y)\|_{L^2(A\times A,K_A\boxtimes K^*_A)}$ as follows:
\begin{align}
& \nonumber \|e^{-t\lambda_{k,0}}\eta_{k,0}(x)\otimes \eta_{k,0}^*(y)-e^{-t\lambda_{k,z_n}}\eta_{k,z_n}(x)\otimes \eta_{k,z_n}^*(y)\|_{L^2(A\times A,K_A\boxtimes K^*_A)}\leq\\
&\nonumber \|e^{-t\lambda_{k,0}}\eta_{k,0}(x)\otimes \eta_{k,0}^*(y)-e^{-t\lambda_{k,0}}\eta_{k,z_n}(x)\otimes \eta_{k,z_n}^*(y)\|_{L^2(A\times A,K_A\boxtimes K^*_A)}+ \\
& \nonumber \|e^{-t\lambda_{k,0}}\eta_{k,z_n}(x)\otimes \eta_{k,z_n}^*(y)-e^{-t\lambda_{k,z_n}}\eta_{k,z_n}(x)\otimes \eta_{k,z_n}^*(y)\|_{L^2(A\times A,K_A\boxtimes K^*_A)}=\\
& \nonumber e^{-t\lambda_{k,0}}\|\eta_{k,0}(x)\otimes \eta_{k,0}^*(y)-\eta_{k,z_n}(x)\otimes \eta_{k,z_n}^*(y)\|_{L^2(A\times A,K_A\boxtimes K^*_A)}+|e^{-t\lambda_{k,0}}-e^{-t\lambda_{k,z_n}}|=\\
&\nonumber e^{-t\lambda_{k,0}}\sqrt{2-2\langle\eta_{k,0},\eta_{k,z_n}\rangle^2_{L^2\Omega^{m,0}(A,h|_A)}}+|e^{-t\lambda_{k,0}}-e^{-t\lambda_{k,z_n}}|.
\end{align}
Altogether we have shown that $$\|K_{z_n}(t,x,y)-K_0(t,x,y)\|_{L^2(A\times A,K_A\boxtimes K^*_A)}\leq \sum_{k=1}^{\overline{k}}e^{-t\lambda_{k,0}}\sqrt{2-2\langle\eta_{k,0},\eta_{k,z_n}\rangle^2_{L^2\Omega^{m,0}(A,h|_A)}}+\sum_{k=\ell+1}^{\overline{k}}|e^{-t\lambda_{k,0}}-e^{-t\lambda_{k,z_n}}|+\epsilon$$
where, likewise the previous cases, $\ell:=\dim(\ker(\Delta_{\overline{\partial},m,0,\mathrm{abs}}))$.
Arguing as in the proof of Th. \ref{tracenorm} we can find a sufficiently big integer $\overline{n}>0$ such that for any $n>\overline{n}$, $k=0,...,\overline{k}$ and $t\in [t_0,\infty)$ we  have $$\sum_{k=1}^{\overline{k}}e^{-t\lambda_{k,0}}\sqrt{2-2\langle\eta_{k,0},\eta_{k,z_n}\rangle^2_{L^2\Omega^{m,0}(A,h|_A)}}+\sum_{k=\ell+1}^{\overline{k}}|e^{-t\lambda_{k,0}}-e^{-t\lambda_{k,z_n}}|\leq \epsilon.$$
So we have just proved that for any sequence $\{s_n\}_{n\in\mathbb{N}}\subset (0,1]$ with $s_n\rightarrow 0$ as $n\rightarrow \infty$ there exists a subsequence $\{z_n\}_{n\in\mathbb{N}}\subset \{s_n\}_{n\in\mathbb{N}}$ such that for any arbitrarily fixed $0<\epsilon<1$ and $t_0>0$ there exists a positive integer $\overline{n}$ such that 
$$\|K_{z_n}(t,x,y)-K_0(t,x,y)\|_{L^2(A\times A,K_A\boxtimes K^*_A)}\leq 2\epsilon$$ for any $n>\overline{n}$ and $t\in [t_0,\infty)$.
In other words  for any sequence $\{s_n\}_{n\in\mathbb{N}}\subset (0,1]$ with $s_n\rightarrow 0$ as $n\rightarrow \infty$ there exists a subsequence $\{z_n\}_{n\in\mathbb{N}}\subset \{s_n\}_{n\in\mathbb{N}}$ such that for any arbitrarily fixed $t_0>0$ we have  $$\lim_{n\rightarrow \infty} \|K_{z_n}(t,x,y)-K_0(t,x,y)\|_{L^2(A\times A,K_A\boxtimes K^*_A)}=0$$  uniformly on $[t_0,\infty)$. We can thus conclude  that for any arbitrarily fixed $t_0>0$  $$\lim_{s\rightarrow 0} \|K_{s}(t,x,y)-K_0(t,x,y)\|_{L^2(A\times A,K_A\boxtimes K^*_A)}=0$$  uniformly on $[t_0,\infty)$ as desired.

\end{proof}

\section{Examples and applications}

This last section is devoted to some examples and applications. First of all we want to show that Hermitian pseudometrics appear naturally when we deal with singular complex projective varieties endowed with the Fubini-Study metric and more generally when we consider  compact and irreducible Hermitian complex spaces.  Complex  spaces are a classical topic in complex geometry and we refer to \cite{GeFi} and \cite{GrRe} for an in-depth treatment. Here we recall that an irreducible complex space $X$ is a reduced complex space such that $\reg(X)$, the regular part of $X$, is connected. Furthermore we recall  that a paracompact  and reduced complex space $X$ is said \emph{Hermitian} if  the regular part of $X$ carries a Hermitian metric $h$  such that for every point $p\in X$ there exists an open neighborhood $U\ni p$ in $X$, a proper holomorphic embedding of $U$ into a polydisc $\phi: U \rightarrow \mathbb{D}^N\subset \mathbb{C}^N$ and a Hermitian metric $g$ on $\mathbb{D}^N$ such that $(\phi|_{\reg(U)})^*g=h$, see for instance \cite{Ohsa} or \cite{JRu}. In this case we will write $(X,h)$ and with a little abuse of language we will say that $h$ is a Hermitian metric on $X$. Clearly any analytic subvariety of a complex Hermitian manifold endowed with the metric induced by the restriction of the  metric of the ambient space is a Hermitian complex space. In particular, within this class of examples, we have any complex projective variety $V\subset \mathbb{C}\mathbb{P}^n$ endowed  with the K\"ahler metric induced by the Fubini-Study metric of $\mathbb{C}\mathbb{P}^n$.  As showed by a very deep result due to Hironaka, the singularities of a complex space can be resolved. We refer to the celebrated work of Hironaka \cite{Hiro}, to \cite{Aroca}, \cite{BiMi} and \cite{HH} for a thorough discussion  on this subject. Furthermore we refer to \cite{GMMI} and \cite{MaMa}  for a quick introduction. Below we simply provide a very brief account with the material that is  strictly necessary for our purposes. Let $X$ be a compact and irreducible complex space. Then there exists a compact complex manifold $M$, a divisor with only normal crossings $D\subset M$ and a surjective holomorphic map $\pi:M\rightarrow X$ such that $\pi^{-1}(\sing(X))=D$ and 
\begin{equation}
\label{hiro}
\pi|_{M\setminus D}: M\setminus D\longrightarrow X\setminus \sing(X)
\end{equation}
is a biholomorphism. Assume now that $(X,h)$ is a compact and irreducible Hermitian complex space. Then, by the very definition of Hermitian complex space, it is immediate to deduce that $\pi^*h$ extends smoothly on the whole $M$ as a positive semidefinite Hermitian product strictly positive on $M\setminus D$. In other words $\pi^*h$ becomes a Hermitian pseudometric on $M$ whose degeneracy locus $Z$ is contained in $D$.\\ 
Now we continue with the next proposition that provides a quite general situation to which the theorems of the previous sections apply.

\begin{prop}
\label{smoothdeg}
Let $(M,J)$ be a compact complex manifold of complex dimension $m$. Let $p:M\times [0,1]\rightarrow M$ be the canonical projection and let $g_s\in C^{\infty}(M\times [0,1], p^*T^*M\otimes p^*T^*M)$ be a smooth section of $p^*T^*M\otimes p^*T^*M\rightarrow M\times[0,1]$ such that:
\begin{enumerate}
\item $g_s(JX,JY)=g_s(X,Y)$ for any $X,Y\in \mathfrak{X}(M)$ and $s\in [0,1]$;
\item $g_s$ is a Hermitian metric on $M$ for any $s\in (0,1]$;
\item $g_0$ is symmetric, positive semidefinite and positive definite over $A$, where $A\subset M$ is open and dense;
\item $(A,g_1|_A)$ is parabolic;
\item There exists a positive constant $\frak{a}$ such that $g_0\leq \frak{a}g_s$ for each $s\in [0,1]$.
\end{enumerate}
Then there exists also a positive constant $\mathfrak{b}$ such that $g_s\leq\mathfrak{b} g_1$ for each $s\in [0,1]$ and thus Th. \ref{spectralth},\ref{tracenorm},\ref{Zeta} and \ref{HilSCon} apply to $(M,g_s)$ and $(A,h|_A)$ with $h:=g_0$.
\end{prop}

The proof of the above proposition will be an immediate consequence of the next lemma.
Let $F_s\in C^{\infty}(M\times [0,1],p^*\mathrm{End}(TM))$ be such that $g_1(F_s\cdot,\cdot)=g_s(\cdot,\cdot)$ for each $s\in [0,1]$. 

\begin{lemma}
\label{garara}
For any $s\in [0,1]$ and $p\in M$ let $|F|_{g_1}:M\times [0,1]\rightarrow \mathbb{R}$ be defined by $(p,s)\mapsto |F_s|_{g_1}(p)$,  where $|F_s|_{g_1}:M\rightarrow \mathbb{R}$ is the pointwise operator norm of $F_s$ with respect to $g_1$, see  \eqref{pointnorm} for the definition. Then we have the following properties:
\begin{enumerate}
\item  The function $|F|_{g_1}:M\times [0,1]\rightarrow \mathbb{R}$  is continuous.
\item The function $\||F_s|_{g_1}\|_{L^{\infty}(M)}:[0,1]\rightarrow \mathbb{R}$, defined by $s\mapsto \||F_s|_{g_1}\|_{L^{\infty}(M)}$, is continuous.
\item For each $s\in [0,1]$ we have 
$$g_s\leq \mathfrak{b}g_1$$ with
$$\mathfrak{b}=\max_{s\in [0,1]}\||F_s|_{g_1}\|_{L^{\infty}(M)}$$
\end{enumerate}
\end{lemma}

\begin{proof}
Let $p\in M$ and $s\in [0,1]$ be arbitrarily fixed.  Let 
\begin{equation}
\label{autovalorih}
\{\lambda_{1,s}(p),\lambda_{1,s}(p),\lambda_{2,s}(p),\lambda_{2,s}(p),...,\lambda_{m,s}(p),\lambda_{m,s}(p)\}
\end{equation}
 with $0\leq \lambda_{1,s}(p)\leq \lambda_{2,s}(p)\leq...\leq \lambda_{m-1,s}(p)\leq \lambda_{m,s}(p)$, be the eigenvalues of  $F_{s,p}:T_pM\rightarrow T_pM$. Then it is well known that 
\begin{equation}
\label{popoh}
|F_s|_{g_1}(p)=\sup_{0\neq v\in T_pM}\frac{g_1(F_{s,p}v,v)}{g_1(v,v)}=\lambda_{m,s}(p).
\end{equation}
Now we observe that, for any $k=1,2,..,m$, the function $\lambda_{k}:M\times [0,1]\rightarrow \mathbb{R}$ defined as  $\lambda_{k}(p,s):=\lambda_{k,s}(p)$ is continuous. This follows easily by the fact that $F_s$ is given in local coordinates by a real, symmetric, square matrix of rank $2m\times 2m$ whose entries are continuous (actually smooth) functions of $(2m+1)$-variables. It is in fact a classical result of linear algebra that the eigenvalues  of a real symmetric, square matrix $M=(a_{i,j})$,  whose entries $a_{i,j}:W\rightarrow \mathbb{R}$ are continuous functions defined over an open subset $W\subset \mathbb{R}^{\ell}$, are themselves continuous functions over $W$. We can thus conclude that $\lambda_{m}:M\times [0,1]\rightarrow \mathbb{R}$ is continuous and eventually this tells us that also $$[0,1]\ni s\mapsto \max_{p\in M}\lambda_{m,s}(p)\in \mathbb{R}$$ is continuous. As $\||F_s|_{g_1}\|_{L^{\infty}(M)}=\max_{p\in M}\lambda_{m,s}(p)$ for any $s\in [0,1]$ the first two points are thus established. The third point is now an immediate consequence. Namely given any $p\in M$ and $v\in T_pM$ we have $g_s(v,v)=g_1(F_sv,v)\leq \||F_s|_{g_1}\|_{L^{\infty}(M)}g_1(v,v)\leq \mathfrak{b}g_1(v,v)$.
\end{proof}

Finally we conclude with the following family of examples.

\begin{prop}
\label{example}
Let $(M,J)$ be a compact complex manifold, $D\subset M$  a normal crossing divisor, $h$  a Hermitian pseudometric on $M$ positive definite on $M\setminus D$ and $g$ a Hermitian metric on $M$. If  $f(s)$ is a smooth function on $[0,1]$ such that $f(0)=0$, $f(1)=1$ and $0<f(s)\leq1$ for $s\in (0,1)$ then $g_s:=(1-f(s))h+f(s)g$ satisfies the requirements of  Prop. \ref{smoothdeg}. 
\end{prop}
\begin{proof}
Obviously $g_s$ is compatible with $J$, it is a Hermitian metric whenever $s>0$ and it is a positive semidefinite Hermitian metric when $s=0$. Moreover since $D$ is a finite union of compact complex submanifolds it is known that $M\setminus D$ is parabolic with respect to any Riemannian metric on $M$, see \cite[Prop. 4.5]{BeGu}. Now  let $\frak{b}\in \mathbb{R}$ such that $h\leq \frak{b}g$ and let $\frak{a}=\frak{b}+1$.
We claim that $h\leq \frak{a}g_s$. In fact $h\leq \frak{a} g_s$ if and only if $0\leq \frak{a}h-\frak{a}f(s)h+\frak{a}f(s)g-h$ that is $0\leq \frak{b}h+h-\frak{b}f(s)h-f(s)h+\frak{a}f(s)g-h$ which in turn is equivalent to $0\leq \frak{b}(1-f(s))h+f(s)(\frak{a}g-h)$. Finally  it is immediate to check that this last inequality holds true.
\end{proof}

\end{document}